\documentclass[reqno,12pt,letterpaper]{amsart}
\usepackage{amsmath,amssymb,amsthm,graphicx,mathrsfs,url,bbm,soul}
\usepackage[usenames,dvipsnames]{color}
\usepackage[colorlinks=true,linkcolor=Red,citecolor=Green]{hyperref}
\usepackage{amsxtra}
\usepackage{tikz}

\def\?[#1]{\textbf{[#1]}\marginpar{\Large{\textbf{??}}}}

\let\epsilon=\varepsilon 

\newtheorem{thm}{Theorem}
\newtheorem{prop}{Proposition}[section]

\newtheorem{lem}[prop]{Lemma}
\newtheorem{cor}[prop]{Corollary}

\numberwithin{equation}{section}

\DeclareMathOperator{\ad}{ad}

\DeclareMathOperator{\Spec}{Spec}

\DeclareMathOperator{\Op}{Op}

\DeclareMathOperator{\SL}{SL}
\DeclareMathOperator{\supp}{supp}

\newcommand{\Sc}{\mathcal S}

\newcommand{\R}{\mathbb R}

\newcommand{\N}{\mathbb N}
\newcommand{\Z}{\mathbb Z}
\newcommand{\C}{\mathbb C}
\newcommand{\T}{\mathbb T}
\newcommand{\One}{\mathbbm{1}}

\newcommand{\cinf}{C^{\infty}}
\newcommand{\dd}{\partial}

\newcommand{\ls}{\lesssim}

\newcommand{\la}{\langle}
\newcommand{\ra}{\rangle}

\newcommand{\Ran}{\mathrm{Ran}}
\newcommand{\spn}{\mathrm{span}}

\newcommand{\Ohi}{O(h^{\infty})}
\newcommand{\Sp}{\mathrm{Sp}}
\newcommand{\wh}{\widehat}

\title{The Spectrum of an Almost Maximally Open Quantized Cat Map}
\author{Yonah Borns-Weil}
\email{yonah\_borns-weil@berkeley.edu}
\address{Department of Mathematics, University of California, Berkeley,
CA 94720}
\keywords{}
\subjclass[]{}

\begin{document}

\begin{abstract}
We consider eigenvalues of a quantized cat map (i.e. hyperbolic symplectic integer matrix), cut off in phase space to include a fixed point as its only periodic orbit on the torus. We prove a simple formula for the eigenvalues on both the quantized real line and the quantized torus in the semiclassical limit as $h\to0$. We then consider the case with no fixed points, and prove a superpolynomial decay bound on the eigenvalues. The results are illustrated with numerical calculations.
\end{abstract}
\maketitle

\section{Introduction}



The study of the classical-quantum correspondence dynamics consists of comparing quantum evolution to classical dynamical systems. One example of this correspondence  is the relationship between continuous-time classical dynamics and the theory ofscattering resonances – see the textbook by Dyatlov and Zworski \cite{DZ} for an introduction. Recently, there has been interest in the discrete-time analogue given by \emph{open quantum maps}. These maps correspond to quantizations of symplectomorphisms on a given symplectic manifold, which are ``opened" by designating some of the points to ``escape to infinity". Of particular interest are maps corresponding to \emph{hyperbolic} symplectomorphisms, where one may study the relationship between quantum and classical dynamics near a hyperbolic trapped set. See the paper by Nonnenmacher \cite{Nonnenmacher} for a general overview of the literature.

Perhaps the most fundamental example in classical hyperbolic dynamics is the generalized \emph{Arnol'd cat map} given by a matrix in $\SL(2,\Z)=\Sp(2,\Z)$, which act on both $\R^2$ and $\T^2$ and have a single hyperbolic fixed point at the origin. The natural quantum-dynamical analogue is a class of metaplectic operators, known as \emph{quantized cat maps}. On the quantized torus, such maps were treated by Bouzouina and De Bi\'{e}vre \cite{BD}, who proved a version of quantum ergodicity in this setting (see also Dyatlov–J\'{e}z\'{e}quel \cite{DJ} and Faure–Nonnenmacher–De Bi\'{e}vre \cite{FND}). These quantum cat maps can then be ``opened" as part of the general framework of open quantum maps, which corresponds to sending some points ``to infinity" to mimic scattering. Once adjusted in this manner, such maps enjoy many similar results to continuous scattering problems, and are a subject of current research. In the physics literature, open quantum maps are used as a toy computational model for quantum chaos, see Saraceno–Vallejos \cite{SV} and Nonnenmacher–Zworski \cite{NZ} and references given there. In addition, in Nonnenmacher, Sj\"ostrand, and Zworski \cite{NSZ2} reduce the study of some continuous-time systems that op open quantum maps. Schenck \cite{Schenck} considers the similar case of partially open quantum maps, which gives a discrete-time analogue to a damped wave equation.

In this paper, we consider a distinct but related problem, of determining the quantum behavior of a cat map very near its hyperbolic fixed point. Specifically, suppose $M$ is a cat map and $\wh{M}$ is its quantization, acting either on the quantized real line or the quantized torus. We may consider $\Op_h\chi\wh{M}$ or $\Op_N\chi\wh{M}_N$ respectively, where $\chi$ is a bump function equal to $1$ near the origin and $\Op$ refers to semiclassical quantization, see Sections \ref{ss:microlocal} and \ref{ss:microlocalHN}. This has the effect of isolating the behavior near the hyperbolic fixed point, and producing a compact operator whose spectrum can be studied.

The case of such an ``almost maximally open" map has been studied before in the case of Schr\"odinger operators on $\R^n$. The first results in the complicated setting of obstacle problems were provided by Bardos–Guillot–Ralston \cite{BGR}, and the problem was settled by Ikawa \cite{Ikawa} with improvements by Gerard \cite{Gerard}. The semiclassical case was then handled by Gerard–Sj\"ostrand \cite{GS} and Sj\"ostrand \cite{Sjostrand}, see also Louati and Rouleux \cite{LR}. Surprisingly, to the author's knowledge, the analogous situation for an open quantum map has never been studied before. We shall show that the situations on the torus and on the real line are similar, which matches our intuition that everything away from a small local region had ``escaped". When localized by a small enough microlocal cutoff operator and for small enough $h$, the spectrum is discrete away from the origin, and the eigenvalues are given up to $\Ohi$ error by $\lambda^{-\frac{2k+1}{2}}$, where $\lambda$ is the positive eigenvalue of the underlying symplectic transformation.

In addition, we consider the ``maximally open" case of $\Op_N\nu\wh{M}_N$, where $\nu$ is a bump function away localized in a small annulus about the hyperbolic fixed point. This map corresponds to a discrete-time analogue of the ``nontrapping estimates" seen for scattering resonances in Section 6.4 of \cite{DZ}. In this case, the eigenvalues decay to zero, and we shall show in Section \ref{s:rates} that they in fact decay to zero like $\Ohi$.

\subsection{The Classical Cat Map}\label{ss:catmap}
The original Arnol'd cat map was defined as the matrix $$M=\begin{pmatrix}2&1\\1&1\end{pmatrix}$$ acting on $\R^2$, and more interestingly also acting on the two-dimensional torus $\T^2$. As $M\in\mathrm{SL}(2,\Z)$, one sees that $M$ is invertible and measure-preserving with respect to the Lebesgue measure. When considered as a measure-preserving map on $\T^2$, $M$ is mixing, which in particular implies that it is ergodic. In addition, $M$ is topologically transitive and has a dense set of periodic points, which gives it to be \emph{chaotic} as defined by e.g. Teschl \cite{Teschl}.

These above properties hold for any $M\in\mathrm{SL}(2,\Z)$ with eigenvalues $\lambda,\frac{1}{\lambda}\not\in S^1$, so we call such matrices \emph{generalized cat maps}, or for convenience just \emph{cat maps}. We discuss the set of fixed and periodic points in further detail. By an argument using the pigeonhole principle, we may see that points on $\T^2$ are periodic points of $M$ if and only if they have rational coordinates. We now provide an elementary proposition about periodic orbits that will motivate our main result.
\begin{prop}
Let $M=Q^{-1}DQ$ be a generalized cat map with $D$ diagonal and $\det Q=1$. Then the only periodic orbit of $M$ on $\T^2$ contained in $B_{\frac{1}{4\lambda\|Q\|^2}}(0)$ is the fixed point at $0$.
\end{prop}

\begin{proof}
We begin by proving this for symmetric $M$, which is equivalent to $\|Q\|=1$. Write $v\in B_{\frac{1}{4\lambda}}(0)$ as $v=a_1v_1+a_2v_2$, where $v_1, v_2$ are the normalized eigenvectors of $M$ acting on $\R^2$ corresponding to $\lambda, \lambda^{-1}$, so \begin{equation}\label{Mnv}M^nv=\lambda^na_1v_1+\lambda^{-n}a_2v_2.\end{equation} If $v$ is a periodic point, we can easily see that if one of $a_1$ is $0$ then $a_2$ must also be, so without loss of generality assume $a_1,a_2\neq 0$, and note that $|a_1|,|a_2|\le\frac{1}{4\lambda}$. Let $n=\left\lfloor \log_{\lambda}\frac{1}{4|a_1|}\right\rfloor$, so \begin{equation}\label{lambda1}\frac{1}{4\lambda}<\lambda^n|a_1|\le \frac{1}{4}\end{equation} and also \begin{equation}\label{lambda2}|\lambda^{-n}a_2|\le |a_2|\le \frac{1}{4\lambda}.\end{equation} Combining (\ref{lambda1}) and (\ref{lambda2}) with (\ref{Mnv}) gives that $\frac{1}{4\lambda}<\|M^nv\|\le\frac{1}{2}$, which proves the result for this case.

In the general case, we let $M=Q^{-1}DQ$. Let $v\in Q^{-1}B_{\epsilon}(0)$. Then after scaling, the proof for the symmetric case says that there is an $n\in\N$ such that $D^nQv\in B_{2\lambda\epsilon}(0)\setminus B_{\epsilon}(0)$, or equivalently that $M^nv\in Q^{-1}B_{2\lambda\epsilon}(0)\setminus Q^{-1}B_{\epsilon}(0)$. Then this gives that any point int $B_{\frac{\epsilon}{\|Q\|}}(0)$ will be taken to $Q^{-1}B_{2\lambda\epsilon}(0)\setminus Q^{-1}B_{\epsilon}(0)$ which is itself contained in $B_{2\lambda\epsilon\|Q\|}(0)$. Letting $\epsilon=\frac{1}{4\lambda\|Q\|}$ gives the desired result.
\end{proof}

\subsection{Overview of Theorems and Proofs}
\label{ss:thms}

We work in the setting of operators on a quantum ``state space," which is a Hilbert space on which ``quantized" observables act. In the case of $T^*\R$, the state space is simply $L^2(\R)$, while in the case of the torus $\T^2$, it is a space $H_N$ which is $N$-dimensional and spanned by the Dirac comb distributions $\frac{1}{\sqrt{N}}\sum_{k\in\Z}\delta_{x=k+\frac{n}{N}}$. Observables are given by semiclassical pseudodifferential operators $\Op_h\chi$ (on the real line) and $\Op_N\chi$ (on the torus). From that framework we can then define metaplectic operators $\wh{M}$, $\wh{M_N}$ that ``quantize" linear symplectic transformations, which on $\T^2$ are precisely the cat maps discussed in Section \ref{ss:catmap}.

Our result on the real line is the following:

\begin{thm}\label{mainL2}
Let $\wh M$ be a metaplectic operator with eigenvalues $\lambda, \lambda^{-1}$ of $M$ where $\lambda>1$, and let $\chi\in\cinf_0(\R^2)$ equal to $1$ near the origin. Then for all $\delta>0$ there is an $h_0>0$ such that if $0<h<h_0,$
\begin{equation*}
    \Spec(\Op_h\chi\wh{M})\cap\left\{|z|\ge \delta\right\}=\{\mu_0, \mu_1, \dots\}\cap\left\{|z|\ge \delta\right\},\quad \frac{1}{\mu_k}=\lambda^{\frac{2k+1}{2}}+\Ohi.
\end{equation*}
\end{thm}

On the quantized torus, the analogous result is:
\begin{thm}\label{mainHN}
Let $\wh M_N$ be a metaplectic operator on $H_N$ with eigenvalues $\lambda, \lambda^{-1}$ of $M$ where $\lambda>1$, and let $\chi\in\cinf_0(\T^2)$ equal to $1$ near the origin with support on a small enough neighborhood of $0$. Then for all $\delta>0$, there is an $h_0>0$ such that if $0<h=\frac{1}{2\pi N}<h_0$,
\begin{equation*}
    \Spec(\Op_N\chi\wh{M}_N)\cap\left\{|z|\ge \delta\right\}=\{\mu_0, \mu_1, \dots\}\cap\left\{|z|\ge \delta\right\},\quad \frac{1}{\mu_k}=\lambda^{\frac{2k+1}{2}}+\Ohi.
\end{equation*}
\end{thm}

For exact definitions of the terms in Theorems \ref{mainL2} and \ref{mainHN}, see Sections \ref{ss:microlocal} and \ref{ss:microlocalHN} respectively.

We give an overview of the proofs of Theorems \ref{mainL2} and \ref{mainHN}. There is a strong similarity to the methods used in \cite{Gerard}, and Theorem \ref{mainL2} could be proven using his methods, though we elect to use different underlying function spaces. To find the spectrum of $\Op_h\chi\wh{M}$, it is enough to consider the special case where $M=F$ is a diagonal matrix. Our strategy will be to build a Grushin problem acting on $L^2(\R)\oplus \C^K$ given by
\begin{equation*}
    \mathcal{P}(z):=\begin{pmatrix}I-z\chi(x,hD)\wh{F}&\tilde R_-\\\tilde R_+&0 \end{pmatrix}.
\end{equation*}
Provided $\mathcal{P}(z)$ is invertible with inverse \begin{equation*}
    \mathcal{E}(z):=\begin{pmatrix}E(z)&\tilde E_+(z)\\\tilde E_-(z)&E{-+}(z) \end{pmatrix},
\end{equation*} then $I-z\chi(x,hD)\wh{F}$ is invertible if and only if $E_{-+}$ is, reducing the study of the spectrum to the invertibility of the matrix $E_{-+}$. But $E_{-+}$ is (up to $\Ohi$ errors) conjugate to a simple diagonal matrix, which gives the desired spectrum. To construct the operators $\tilde R_-, \tilde R_+$, we use the now standard approach as outlined by Sj\"{o}strand and Zworski in Section 3.4 of \cite{SZ}. The main difficulty is showing well-posedness; in other words, showing that $\mathcal{P}(z)$ is invertible. This is achieved through careful bounding of $\chi(x,hD)\hat{F}$ on the range of $I-\tilde R_-\tilde R_+$ with respect to an unusual norm, which is different from but fulfills an analogous role to the one seen in \cite{Gerard}. The necessary inequalities are collected and proven in Section \ref{s:prelims}.

To solve the corresponding problem on the quantized torus, we make use of ``sampling operators" $S_{\psi}:L^2(\R)\to H_N$ defined in Section \ref{ss:transplants}. The power of such operators comes from their interplay with others; in particular, $S_{\psi}$ and $S_{\psi}S_{\psi}^*$ act as pseudodifferential operators on $H_N$ as a subset of $\Sc'(\R)$, and hence obey the rules of pseudodifferential calculus. They can then be microlocally cut off to form ``spectral operators" $T:L^2(\R)\to H_N$ and $T:H_N\to L^2(\R)$ which map our Grushin problem on $L^2(\R)$ to one on $H_N$. The study of the spectrum can then be carried out exactly as it was for Theorem \ref{mainL2}.

For completeness, we also provide an analysis of the ``maximally open" or ``nontrapping" case, corresponding to the classical situation of no periodic orbits, such that every path escapes to infinity. It is not surprising that in this case the eigenvalues will decay to $0$ as $h\to 0$. In fact, we may strengthen that result by providing a superpolynomial rate of decay. Though an analogous result should hold on $L^2(\R)$, we only prove this fact on the quantized torus.

\begin{thm}\label{NTDecayRate}
Let $\wh M_N$ be a metaplectic operator on $H_N$ with eigenvalues $\lambda, \lambda^{-1}$ of $M$ where $\lambda>1$, and let $\nu\in\cinf_0(\T^2)$ equal to $0$ near the origin with support on a small enough neighborhood of $0$. Then
\begin{equation*}
    \sup_{\lambda\in\Spec \Op_N\nu\wh{M}_N}|\lambda|=\Ohi.
\end{equation*}
\end{thm}
A discussion and proof of Theorem \ref{NTDecayRate} is given in Section \ref{s:rates}.

Finally, we discuss how small the support of $\chi$ must be in Theorem \ref{mainHN}. We shall show that there is a constant $c$ such the results holds for $\supp\chi,\supp\nu$ respectively contained in a $\frac{c}{\lambda\|Q\|^2}$-neighborhood of the origin, where $Q$ is the matrix of eigenvectors of $M$, normalized to have determinant $1$ (this doesn't quite fix $Q$ but bounds its norm below). The reciprocal dependence on $\lambda$ and $\|Q\|^2$ matches the intuition that we are isolating a single fixed point in a region with no other periodic orbits, as discussed in Section \ref{ss:catmap}.


\subsection{Numerical Experiments}

An additional motivation for working on the finite-dimensional space $H_N$ is the ease by which one may perform numerics. Let $M$ be the original Arnol'd cat map $$M=\begin{pmatrix} 2&1\\1&1\end{pmatrix}$$ and for ease of computation, let $\Op_h\chi=\Op_h^L(\rho(x)\rho(\xi))$ for $\rho\in\cinf(\R)$ and $\Op_h^L$ the left quantization as defined in Sections \ref{ss:microlocal} and \ref{ss:microlocalHN}. By (\ref{Coq}), $\Op_h\chi$ is the quanitzation of a compactly supported bump function up to $\Ohi$ errors. Restricting to $H_N$, we see in Figure \ref{f:Eigs} the largest four eigenvalue real parts of $\Op_N\chi\wh{M}_N$, as functions of $N$. The imaginary parts rapidly decay to $0$ and are omitted from the plot.

\begin{figure}
\includegraphics[scale=0.75]{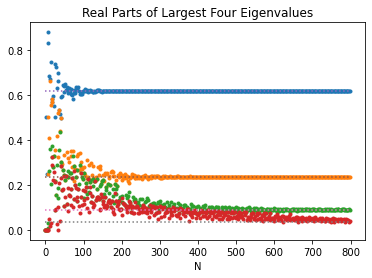}
\caption{Real parts of four largest eigenvalues of $\Op_N\chi\wh{M}_N$ for original Arnol'd cat map $M$ and $\chi=1$ near the origin. The limits predicted by Theorem \ref{mainHN} are given as dotted lines.}
\label{f:Eigs}
\end{figure}

In the nontrapping case, we let $\Op_N\nu=\Op_N^L((\rho(x)-\rho(2x))(\rho(\xi)-\rho(2\xi)))$ as an example of the cutoff operator considered in Theorem \ref{NTDecayRate}. The first plot in Figure \ref{f:NEigs} shows rapid decay to zero of the largest eigenvalue of $\Op_N\nu\wh{M}_N$ on $H_N$, exactly as was predicted by Theorem \ref{NTDecayRate}.

To further show the rate of decay of the spectrum, the second plot in Figure \ref{f:NEigs} shows the norms of largest eigenvalues of $\Op_N\nu\wh{M}_N$ plotted on a log-log plot.
\begin{figure}
    \includegraphics[scale=0.55]{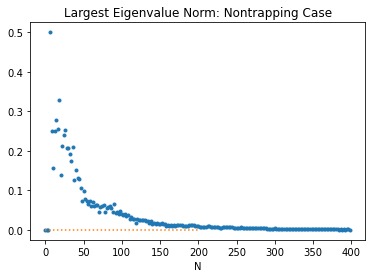}
    \includegraphics[scale=.55]{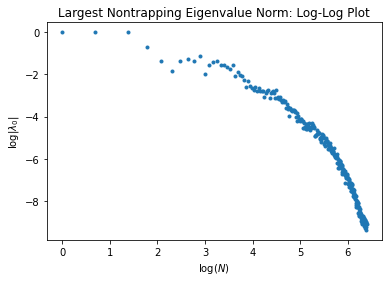}
    \caption{Left: Norm of largest eigenvalue of $\Op_N\nu\wh{M}_N$, with $\nu=0$ near the origin. Right: Log-log plot of norm of the largest eigenvalues of $\Op_N\nu\wh{M}_N$, which illustrates Theorem \ref{NTDecayRate}.}
    \label{f:NEigs}
\end{figure}
We see a decrease to $-\infty$ in the slope of the curve, which points to an $\Ohi$ decay rate of $|\lambda_0|$ consistent with Theorem \ref{NTDecayRate}.

\section{The Operator on the Real Line}
\label{s:L2}

\subsection{Review of Pseudodifferential and Metaplectic Operators}
\label{ss:microlocal}
We recall briefly the semiclassical pseduodifferential and metaplectic calculi on $\R^n$, which will be used repeatedly throughout the paper. The reader may refer to the book by Zworski \cite{Zworski} for a complete introduction. 

A function $m:T^*\R^n\to(0,\infty)$ is called an \emph{order function} if $m(x,\xi)\le C\la |x-y,\xi-\eta|\ra^Nm(y,\eta)$ for some $C,N$. The corresponding \emph{symbol class} of $m$ is defined as $$S(m):=\{a\in\cinf(T^*\R^n):\forall \alpha,\beta\in\Z^n,\, |\dd_x^{\alpha}\dd_{\xi}^{\beta}a(x,\xi)|\le C_{\alpha\beta}m(x,\xi)\}.$$

Note that elements of $S(m)$ (called \emph{symbols}) may depend on $h$, and in such a case the constant $C_{\alpha\beta}$ in the definition must be uniform in $h$.

The \emph{Weyl quantization} of a symbol is an operator $\Op_ha:\Sc(\R^n)\to \Sc'(\R^n)$ (sometimes written $a(x,hD)$) given by \begin{equation*}(\Op_h(a)u)(x):=\frac{1}{(2\pi h)^n}\iint e^{\frac{i}{h}\la x-y,\xi\ra}a\left(\frac{x+y}{2},\xi\right)u(y)\,dy\,d\xi.\end{equation*} These operators are called \emph{semiclassical pseudodifferential operators}.

A classical result of Calderon and Vaillancourt (see \cite{Zworski}) gives that for $a\in S(1)$, $\Op_ha:L^2(\R^n)\to L^2(\R^n)$ is bounded with norm independent of $h$. In addition, if $a$ is real, $\Op_ha$ is a (possibly unbounded) symmetric on $L^2$. In general, the formal adjoint of $\Op_ha$ is $(\Op_ha)^*$.

We may also consider the \emph{left quantization} (also called the \emph{standard quantization}), given by
\begin{equation*}
    (\Op^L_h(a)u)(x):=\frac{1}{(2\pi h)^n}\iint e^{\frac{i}{h}(x-y)\xi}a\left(x,\xi\right)u(y)\,dy\,s\xi.
\end{equation*}
We can go from one quantization to the other via the \emph{change of quantization formula}
\begin{equation*}
    \Op_ha=\Op^L_ha_L
\end{equation*}
with
\begin{equation}\label{Coq}
    a_L=e^{\frac{ih}{2}\la D_x,D_{\xi}\ra}a
\end{equation}

Let $a\in S(m_1), b\in S(m_2)$. The composition of Weyl quantizations is given by the following formula:
\begin{equation*}
    \Op_ha\Op_hb=\Op_h{a\#b}
\end{equation*}
where $a\#b\in S(m_1m_2)$ is the \emph{Moyal product} given by
\begin{equation*}
  a\#b(x,\xi):= e^{ih\sigma(D_x,D_{\xi},D_y,D_{\eta})}(a(x,\xi)b(y,\eta))|_{y=x,\eta=\xi}
\end{equation*}
which gives the asymptotic formula
\begin{equation*}
    a\#b(x,\xi)\sim\sum_{k=0}^{\infty}\frac{i^kh^k}{k!}\sigma(D_x,D_{\xi},D_y,D_{\eta})^k(a(x,\xi)b(y,\eta))|_{y=x,\eta=\xi}
\end{equation*}
in $S(m_1m_2)$.
In particular, if $a,b\in S(1)$ have disjoint supports, then
\begin{equation*}
    \|\Op_h(a)\Op_h(b)\|_{L^2(\R^n)\to L^2(\R^n)}=\Ohi
\end{equation*}
It also follows that if $a,b\in S(1)$, the commutator of quantizations is given by
\begin{equation*}
    [\Op_ha,\Op_hb]=-ih\Op_h(\{a,b\})+O(h^3)_{L^2(\R^n)\to L^2(\R^n)},
\end{equation*}
where $\{\cdot,\cdot\}$ is the Poisson bracket.

Furthermore, if $a\in S$ is bounded below by $c_0>0$, the \emph{G\aa rding inequality} gives that for every $\epsilon>0$ there is a $C>0$ such that for $h$ small enough,
\begin{equation}\label{GI}
    \la \Op_h(a)u,u\ra\ge C(c_0-\epsilon)\|u\|_{L^2(\R^n)}.
\end{equation}
We remark that (\ref{GI}) may be strengthened to the \emph{Sharp G\aa rding inequality} and furthermore to the \emph{Fefferman-Phong inequality}, which give negative $h$-dependent lower bounds when $\epsilon=0$ (see Theorem 4.32 in \cite{Zworski}), though we will not need such bounds for this paper.

The exponential of a pseudodifferential operator is also a pseudodifferential operator. If $g=\log m$ for order function $m$, then $e^{t\Op_hg}$ is the semiclassical quantization of some $b_t\in S(m^t)$. In particular, if $g\in\cinf_0(T^*\R^n)$, then there is a $C$ such that
\begin{equation*}
    \|e^{t\Op_hg}\|_{L^2(\R^n)\to L^2(\R^n)}\le e^{C|t|}.
\end{equation*}

Finally, we discuss the slightly less standard theory of functional calculus applied to Weyl pseudodifferential operators.

\begin{lem}\label{RnFunCalc}
If $a\in S(1)$ or $a\in S(m)$ is elliptic in the sense that $a\ge \frac{1}{C}m-C$, and $\phi\in\cinf_0(T^*\R^n)$, then $\phi(\Op_ha)$ is a Weyl pseudodiffernetial operator with symbol in $S(m^{-k})$ for every $k$. Suppose $b\in S^1$ such that $\supp b\cap a^{-1}(\supp\phi)=\emptyset$. Then $\phi(\Op_h a)\Op_hb=\Ohi_{L^2(\R^n)\to L^2(\R^n)}$.
\end{lem}

The proof of Lemma \ref{RnFunCalc} follows from the Helffer-Sj\"ostrand formula ((8.13) in \cite{DS}), together (8.15–8.16) in \cite{DS} which provide the asymptotic expansion for the symbol of $\phi(\Op_ha)$ as 
$\sum_{j=0}^{\infty}p_j$ with
\begin{equation*}
    p_j(x,\xi)=\frac{1}{(2j)!}\dd_t^{2j}\left(q_j(x,\xi,t)\phi(t)\right)|_{t=a_0(x,\xi)},
\end{equation*}
where $a_0=\lim_{h\to0}a$ and $q_j$ are symbols coming from terms in the semiclassical expansion of the resolvent $(z-\Op_ha)^{-1}$.

We now discuss the calculus of metaplectic operators. Let $M=\begin{pmatrix}A&B\\C&D \end{pmatrix}\in Sp(2n,\R)$ be a real symplectic matrix. Then the \emph{semiclassical metaplectic operators associated to $M$} are $\wh{M}:L^2(\R^n)\to L^2(\R^n)$ given in \cite{Zworski}\cite{DJ}. In the special case when $\det A\neq 0$, we have (up to a global phase) that
\begin{equation*}
    \wh{M}u(y)=\frac{|\det A|^{-\frac{1}{2}}}{(2\pi h)^n}\iint e^{\frac{i}{h}\left(\frac{1}{2}\la CA^{-1}x,x\ra+\la A^{-1}x,\eta\ra-\frac{1}{2}\la A^{-1}B\eta,\eta\ra-\la y,\eta\ra\right)}u(y)\,dy\,d\eta.
\end{equation*}

These operators are unitary, and form a representation of the \emph{metaplectic group}, which is a double cover of the symplectic group. The importance of metaplectic operators is that they ``quantize" linear changes of phase variables, in the sense given by the \emph{exact Egorov theorem} \cite{Zworski}\cite{DJ}
\begin{equation*}
    \Op_h(a)\wh{M}=\wh{M}\Op_h(a\circ M).
\end{equation*}

\subsection{Preliminary Inequalities}
\label{s:prelims}

We begin by studying the operator $\chi(x,hD)\wh{F}$ on $\R$, where $\wh{F}u(x)=\frac{1}{\sqrt{\lambda}}u\left(\frac{x}{\lambda}\right)$ and $\chi\in\cinf(T^*\R)$. A similar operator on different spaces was considered by Gerard \cite{Gerard} and on yet different spaces by Faure and Tsujii  \cite{FT}.

For $m\in\N$, we define the following weighted Sobolev-type norm:
\begin{equation*}
    \|u\|_{Y}^2:=\sum_{k=0}^m\|\la x\ra^{-m}u^{(k)}\|_{L^2(\R^n)}^2.
\end{equation*}
Denote by $Y$ the closure of $\cinf_0(\R)$ in the $Y$-norm.

Fix $\Lambda>0$, and define the ``cut-off" version of the $Y$-norm as follows:
\begin{equation*}
    \|u\|_B^2:=\|U_h\mathbbm{1}_{P_h\le\Lambda}u\|_Y^2+\|\mathbbm{1}_{P_h\ge\Lambda}u\|_{L^2(\R)}^2.
\end{equation*}

Here $P_h=P(x,hD)$ is the semiclassical harmonic oscillator $h^2\Delta+x^2$, and $U_hu(x)=h^{\frac{1}{4}}u(h^{\frac{1}{2}}x)$ is a unitary semiclassical rescaling. We remark that the $B$-norm depends on $h,\Gamma$, and $m$, though we generally suppress these constants in writing.

In fact, the $B$-norm is just a dressed up $L^2$ norm, as the following proposition shows.

\begin{prop}\label{NormEquivalence}
There is a $C=C(m, \Lambda)$ such that
\begin{equation*}
    C^{-1}h^{\frac{m}{2}}\|u\|_{L^2(\R)}\le \|u\|_B\le Ch^{-\frac{m}{2}}\|u\|_{L^2(\R)}.
\end{equation*}
\end{prop}
The purpose of Proposition \ref{NormEquivalence} is to be able to write $\Ohi\|u\|_{L^2(\R)}$ and $\Ohi\|u\|_B$ indistinguishably, which we do throughout the remainder of the paper.

\begin{proof}
We first prove the lower bound. By subtracting off $\|\mathbbm{1}_{P_h\ge\Lambda}u\|_{L^2(\R)}^2$ from both sides, we must only show that for some $c>0$,
\begin{equation}\label{FirstIneqReduced}
    \|U_hw\|_Y^2\ge ch^m\|w\|_{L^2(\R)}^2
\end{equation}
for $w=\mathbbm{1}_{P_h\le\Lambda}u$. But
\begin{align*}
    \|U_hw\|_Y^2&=\sum_{k=0}^m\|\la x\ra^{-m}\dd^k(U_hw)\|_{L^2(\R)}^2
    \ge\|\la x\ra^{-m}U_hw\|_{L^2(\R)}^2
    =h^{\frac{1}{2}}\int\la x\ra^{-2m}\left|w\left(h^{\frac{1}{2}}x\right)\right|^2\,dx\\
    &=\int\left\la h^{-\frac{1}{2}}x\right\ra^{-2m}|w(x)|^2\,dx=\left\|\left\la h^{-\frac{1}{2}}x\right\ra^{-m}w\right\|_{L^2(\R)}^2\ge h^m\|\la x\ra^{-m}w\|_{L^2(\R)}^2.
\end{align*}
Now let $\phi,\tilde\phi\in\cinf_0(\{|x|<2\Lambda\})$ such that $\tilde\phi=1$ on $\{|x|\le \Lambda\}$, $\phi=1$ on $\supp\tilde\phi$. Then $w=\tilde\phi(P_h)w$, and we have
\begin{equation*}
    \left\|\la x\ra^{-m}\right\|_{L^2(\R)}^2=\left\|\la x\ra^{-m}\tilde\phi(P_h)w\right\|_{L^2(\R)}^2=\left\|\left(\la x\ra^{-m}\phi(P_h)+I-\phi(P_h)\right)\tilde\phi(P_h)w\right\|_{L^2(\R)}^2.
\end{equation*}

Consider $A=\la x\ra^{-m}\phi(P_h)+I-\phi(P_h)$ as a semiclassical pseudodifferential operator. We have by the Helffer-Sj\"ostrand construction (see \cite{DS} Chapter 8) that $\phi(P_h)$ is a smoothing operator with principal symbol $\phi(x^2+\xi^2)$, so in particular is in $S(1)$, and hence $A$ has $S(1)$ symbol equal to $\la x\ra^{-m}\phi(x^2+\xi^2)+1-\phi(x^2+\xi^2)+O_{L^2\to L^2}(h)$. This symbol is elliptic, so we get for $h$ small enough that:
\begin{equation*}
    \left\|\left(\la x\ra^{-m}\phi(P_h)+I-\phi(P_h)\right)\tilde\phi(P_h)w\right\|_{L^2(\R)}\ge C\|\tilde\phi(P_h)w\|_{L^2(\R)}
    =\|w\|_{L^2(\R)}.
\end{equation*}
Putting the inequalities together gives (\ref{FirstIneqReduced}).

 We now prove the upper bound. As before, by subtracting off $\|\mathbbm{1}_{P_h\ge\Lambda}u\|_{L^2(\R)}^2$ from both sides, we must only show that
\begin{equation}\label{SecondIneqReduced}
    \|U_hw\|_Y^2\le Ch^{-m}\|w\|_{L^2(\R)}^2
\end{equation}
for $w=\mathbbm{1}_{P_h\le\Lambda}u$. But
\begin{align*}
    \|U_hw\|_Y^2=\sum_{k=0}^m\left\|\la x\ra^{-m}\dd^k(U_hw)\right\|_{L^2(\R)}^2&\le\sum_{k=0}^m\left\|\dd^k(U_hw)\right\|_{L^2(\R)}^2=\sum_{k=0}^m\left\|U_hh^{\frac{k}{2}}\dd^kw\right\|_{L^2(\R)}^2\\
    &=\sum_{k=0}^mh^{-k}\left\|(hD)^kw\right\|_{L^2(\R)}^2.
\end{align*}
Let $\phi$ be as before, so $(hD)^kw=(hD)^k\phi(P_h)w$. As $\phi(P_h)$ has symbol $\phi(x^2+\xi^2)\in S(\la\xi\ra^{-k})$, $(hD)^k\phi(P_h)$ has symbol in $S(1)$ and hence is bounded on $L^2(\R)$. This gives that
\begin{equation*}
    \|U_hw\|_Y^2\le\sum_{k=0}^mh^{-k}\left\|w\right\|_{L^2(\R)}^2\le Ch^{-m}\left\|w\right\|_{L^2(\R)}^2
\end{equation*}
which is (\ref{SecondIneqReduced}).
\end{proof}

The key property of the $Y$ norm is that we can find sharp bounds for $\wh{F}$ on certain closed subspaces. Remark that by Sobolev embedding, functions in $Y$ have $C^{K-1}$ smoothness for $K\le m$, so we may define the subspace
\begin{equation}\label{VKDef}
    V_K:=\{u\in Y:u^{(k)}(0)=0\quad\forall k<K\}.
\end{equation}
The following inequality is inspired by Proposition 3.5 in \cite{Gerard}. Although seemingly elementary, is in some sense the crux of the entire argument to follow.
\begin{lem}\label{YNormEstimate}
There is a $C_0$ such that for $K\le m$, $u\in V_K$,
\begin{equation}\label{Lemma2Ineq}
\|\wh{F}u\|_Y\le C_0\lambda^{-\frac{M}{4}}\|u\|_Y.
\end{equation}
\end{lem}
\begin{proof}
We first write:
\begin{align*}
    \left\|\la x\ra^{-m}\partial_x^k\wh{F}u\right\|_{L^2}^2&=\lambda^{-2k}\left\|\la x\ra^{-m}\wh{F}\partial_x^ku\right\|_{L^2}^2=\lambda^{-2k-1}\int\la x\ra^{-2m}\left|u^{(k)}\left(\frac{x}{\lambda}\right)\right|^2\,dx\\
    &=\lambda^{-2k}\int\la \lambda x\ra^{-2m}|u^{(k)}(x)|^2\,dx.
\end{align*}

When $k\ge K$, we then directly see that $$\left\|\la x\ra^{-m}\partial_x^k\wh{F}u\right\|_{L^2}^2\le \lambda^{-2K}\|\la x\ra^{-m}u^{(k)}\|_{L^2}^2.$$

When $k<K$, we require an additional argument.
\begin{lem}\label{sobolev_bound}
If $r<1$, $u\in V_K$, $0\le k<K$, $$\|u^{(k)}\|_{L^2([-r,r])}\le \frac{2\sqrt{2}r^{K-k+\frac{1}{2}}}{(K-k)!}\left(\|u^{(K)}\|_{L^2([-1,1])}+\|u^{(K+1)}\|_{L^2([-1,1])}\right)$$
\end{lem}
\begin{proof}[Proof of Lemma \ref{sobolev_bound}]
We easily see that
\begin{equation}\label{in0}
\|u^{(k)}\|_{L^2([-r,r])}\le\sqrt{2r}\|u^{(k)}\|_{L^{\infty}([-r,r])}.
\end{equation}
But 
\begin{align*}u^{(k)}(x)&=\frac{1}{(K-k-1)!}\int_0^xu^{(K)}(t)(x-t)^{K-k-1}\,dt\\&=\frac{x^{K-k}}{(K-k-1)!}\int_0^1u^{(K)}(sx)(1-s)^{K-k-1}\,ds.
\end{align*}
This gives that
\begin{equation}\label{in1}
\|u^{(k)}\|_{L^{\infty}([-r,r])}\le\frac{r^{K-k}}{(K-k)!}\|u^{(K)}\|_{L^{\infty}([-r,r])}.
\end{equation}
But by the Sobolev Inequality, 
\begin{equation}\label{in2}
\|u^{(K)}\|_{L^{\infty}([-r,r])}\le 2\left(\|u^{(K)}\|_{L^2([-1,1])}+\|u^{(K+1)}\|_{L^2([-1,1])}\right).
\end{equation}

Combining (\ref{in0}), (\ref{in1}), and (\ref{in2}) gives the inequality.
\end{proof}

Fix $s=s(\lambda)\in\left(1,\sqrt{\frac{\lambda^2+1}{2}}\right)$ to be chosen later, and let $r=\sqrt{\frac{s^2-1}{\lambda-s^2}}<1$, which is chosen such that $x> r\implies \la \lambda x\ra>s\la x\ra$. Then
\begin{align*}
    \left\|\la x\ra^{-m}\dd_x\wh{F}u\right\|_{L^2}^2&=\lambda^{-2k}\int_{|x|\le r}\la \lambda x\ra^{-2m}|u^{(k)}(x)|^2\,dx+\lambda^{-2k}\int_{|x|>r}\la \lambda x\ra^{-2m}|u^{(k)}(x)|^2\,dx.\\
    &\le \lambda^{-2k}\int_{|x|\le r}\la \lambda x\ra^{-2m}|u^{(k)}(x)|^2\,dx+\lambda^{-2k}s^{-2m}\int_{|x|>r}\la x\ra^{-2m}|u^{(k)}(x)|^2\,dx.
\end{align*}

But by Lemma \ref{sobolev_bound} the first term is
\begin{align*}
    \lambda^{-2k}\int_{|x|\le r}\la \lambda x\ra^{-2m}&|u^{(k)}(x)|^2\,dx\le \lambda^{-2k}\int_{|x|\le r}|u^{(k)}(x)|^2\,dx\\
    &\le \lambda^{-2k}\left(2\sqrt{2}\frac{r^{K-k+\frac{1}{2}}}{(K-k)!}\left(\|u^{(K)}\|_{L^2([-1,1])}+\|u^{(K+1)}\|_{L^2([-1,1])}\right)\right)^2\\
    &\le 16\lambda^{-2k}\frac{r^{2K-2k+1}}{(K-k)!^2}\left(\|u^{(K)}\|_{L^2([-1,1])}^2+\|u^{(K+1)}\|_{L^2([-1,1])}^2\right).
\end{align*}

We now wish to optimize $r$ and $s$ to get the best possible bound. Set $r=s^{-1}$ and get $s=\lambda^{\frac{1}{4}}$, $r=\lambda^{-\frac{1}{4}}$. Summing everything together gives the bound (\ref{Lemma2Ineq}).
\end{proof}
We also need a bound for a smoothing operator in the $Y$-norm.

\begin{lem}\label{ChiYIneq}
Let $\chi\in\cinf_0(T^*\R)$, and define operator $\widetilde{\chi}$ by $\widetilde{\chi}=\chi(\sqrt{h}x,\sqrt{h}D)$. Then there are $C>0$, $f:\N\to\R^+$ such that 
\begin{equation}\label{ChiYLemIneq}
    \|\widetilde{\chi}u\|_Y\le(C+f(m)\sqrt{h})\|u\|_Y.
\end{equation}
\end{lem}
\begin{proof}
Let $\tilde h=\sqrt{h}$. We compute the commutator of $\widetilde{\chi}$ with a derivative:
\begin{align*}
    [D_x,\widetilde\chi]u(x)&=D_x\iint e^{i(x-y)\xi}\chi(\tilde h\frac{x+y}{2},\tilde h\xi)u(y)\,dy\,d\xi\\&\quad-\iint e^{i(x-y)\xi}\chi(\tilde h\frac{x+y}{2},\tilde h\xi)D_yu(y)\,dy\,d\xi\\
    &=\iint D_x\left(e^{i(x-y)\xi}\chi(\tilde h\frac{x+y}{2},\tilde h\xi)\right)u(y)\,dy\,d\xi\\&\quad+\iint D_y\left(e^{i(x-y)\xi}\chi(\tilde h\frac{x+y}{2},\tilde h\xi)\right)u(y)\,dy\,d\xi\\
    &=\tilde h\iint\left(e^{i(x-y)\xi}\chi'(\tilde h\frac{x+y}{2},\tilde h\xi)\right)u(y)\,dy\,d\xi.
\end{align*}
By an easy induction, we see that
\begin{equation}\label{CommWithDer}
    [D_x^k,\widetilde{\chi}]=\tilde h\chi_k(\tilde hx,\tilde hD)
\end{equation}
where $\chi_k$ is a compactly supported smooth function independent of $h$.

We then have:
\begin{align}
    \nonumber\|\widetilde\chi u\|_Y^2&=\sum_{k=0}^m\|\la x\ra^{-m} D^k\widetilde\chi u\|_{L^2(\R)}^2\\
    \nonumber&=\sum_{k=0}^m\|\la x\ra^{-m} (\widetilde\chi D^k+[D^k,\widetilde\chi]) u\|_{L^2(\R)}^2\\
    \nonumber&\le\sum_{k=0}^m\left(\|\la x\ra^{-m} \widetilde\chi D^k u\|_{L^2(\R)}+\|\la x\ra^{-m} [D^k,\widetilde\chi] u\|_{L^2(
    \R)}\right)^2\\
    \label{TwoTermIneq}&\le 2\sum_{k=0}^m\|\la x\ra^{-m} \widetilde\chi D^k u\|_{L^2(
    \R)}^2+2\tilde h^2\sum_{k=0}^m\|\la x\ra^{-m} \chi_k(\tilde hx,\tilde hD) u\|_{L^2(
    \R)}^2.
\end{align}

Define function $\chi_s(x,\xi)=\chi(\tilde h x,\xi)$, so $\widetilde \chi=\Op_{\tilde h}\chi_s$. Note that while $\chi_s$ is $\tilde h$-dependent, this dependence only contributes to further decay of its derivatives, so it is still in all the same symbol classes as $\chi$. Note also that $\la x\ra^{-m}\in S(\la x\ra^{-m})$, so we have (with $\cdot$ referring to the $x$ variable): 
\begin{align*}
    \la x\ra^{-m}\chi_s(x,\tilde{h}D)&\la x\ra^m=\chi_s(x,\tilde{h}D)+[\la x\ra^{-m},\chi_s(x,\tilde{h}D)]\la x\ra^m\\&=\chi_s(x,\tilde{h}D)+\left(\frac{\tilde{h}}{i}\{\la \cdot\ra^{-m},\chi_s\}(x,\tilde{h}D)+O_{S(\la x\ra^{-m})}(\tilde{h}^3)\right)\la x\ra^m\\
    &=\chi_s(x,\tilde{h}D)+\left(\frac{\tilde{h}}{i}(m\la \cdot\ra^{-m-2}(\cdot)\dd_{\xi}\chi_s)^w(x,\tilde{h}D)+O_{S(\la x\ra^{-m})}(\tilde{h}^3)\right)\la x\ra^m\\
    &=\chi_s(x,\tilde{h}D)+\left(O_{S(\la x\ra^{-m})}(\tilde{h})\right)\la x\ra^m.
\end{align*}

Then $\la x\ra^{-m}\chi_s(x,\tilde{h}D)\la x\ra^m$ is bounded by $C+O(\tilde h)$ from $L^2(\R)$ to $L^2(
\R)$ where the big-O depends on $m$ but the $C$ does not. This proves that:
\begin{equation}\label{CommIneq}
    \|\la x\ra^{-m} \widetilde\chi D^k u\|_{L^2(\R)}\le(C+O(\tilde h))\|\la x\ra^{-m} D^k u\|_{L^2(\R)}.
\end{equation}
Similarly we see that
\begin{equation*}
\la x\ra^{-m}\chi_{k,r}(x,\tilde{h}D)\la x\ra^m\le\chi_{k,r}(x,\tilde{h}D)+\left(O_{S(\la x\ra^{-m})}(\tilde{h})\right)\la x\ra^m
\end{equation*}
where $\chi_{k,r}(x,\xi)=\chi_k(\tilde h x, \xi)$, which gives
\begin{equation}\label{RemIneq}
    \|\la x\ra^{-m} \widetilde\chi_k(\tilde hx,\tilde h\xi) u\|_{L^2(\R)}\le(C+O(\tilde h))\|\la x\ra^{-m} u\|_{L^2(\R)}.
\end{equation}

Combining (\ref{TwoTermIneq}), (\ref{CommIneq}), and (\ref{RemIneq}) gives (\ref{ChiYLemIneq}).
\end{proof}

We put the pieces together to prove an estimate for $\chi(x,hD)\wh{F}$ in the $B$-norm.

\begin{prop}\label{MainEstimate}

Suppose that $\supp\chi\subseteq\{|x|^2+|\xi|^2< \Lambda/(2\lambda^2)\}$. Then there are $C>0, f:\N\to\R^+$ such that if $u\in Y\cap V_K$, then
\begin{equation}\label{MainEstEq}
    \|\chi(x,hD)\wh{F}u\|_B\le \left(C\lambda^{-\frac{K}{4}}+f(m)\sqrt{h}\right)\|u\|_{B}.
\end{equation}
\end{prop}

We remark that the independence of $C$ from $m$ in Proposition \ref{MainEstimate} (and hence also in Proposition \ref{ChiYIneq}) is crucial, as we will eventually take $K$ very large to make the left hand side of (\ref{MainEstEq}) small, but require that $m
\ge K$ for $V_K$ to even be defined.
\begin{proof}
Recall that
\begin{equation*}
    \|\chi(x,hD)\wh{F}u\|_B^2=\|U_h\One_{P_h\le\Lambda}\chi(x,hD)\wh{F}u\|_Y^2+\|\One_{P_h\ge\Lambda}\chi(x,hD)\wh{F}u\|_{L^2(\R)}^2.
\end{equation*}

By Theorem 6.6 in \cite{Zworski} (or an argument similar to that of Proposition \ref{NormEquivalence}),
\begin{equation*}
    \|\One_{P_h\ge\Lambda}\chi(x,hD)\wh{F}u\|_{L^2(\R)}\le \Ohi\|u\|_{L^2}
\end{equation*}
so by Proposition \ref{NormEquivalence} we must only bound $\|U_h\One_{P_h\le\Lambda}\chi(x,hD)\wh{F}u\|_Y$.

But
\begin{equation}\label{Lemma4Reduced}
    \left\|U_h\One_{P_h\le\Lambda}\chi(x,hD)\wh{F}u\right\|_Y\le\left\|U_h\One_{P_h\ge\Lambda}\chi(x,hD)\wh{F}u\right\|_Y+\left\|U_h\chi(x,hD)\wh{F}u\right\|_Y.
\end{equation}

We bound the first term in (\ref{Lemma4Reduced}) first. We have:
\begin{align*}
    \left\|U_h\One_{P_h\ge\Lambda}\chi(x,hD)\wh{F}u\right\|_Y^2&=\sum_{k=0}^m\left\|\la x\ra^{-m}\dd^k(U_h\One_{P_h\ge\Lambda}\chi(x,hD)\wh{F}u)\right\|_{L^2(\R)}^2\\
    &\le\sum_{k=0}^m\left\|\dd^k(U_h\One_{P_h\ge\Lambda}\chi(x,hD)\wh{F}u)\right\|_{L^2(\R)}^2\\
    &=\sum_{k=0}^mh^{-k}\left\|(hD)^k(\One_{P_h\ge\Lambda}\chi(x,hD)\wh{F}u)\right\|_{L^2(\R)}^2
\end{align*}

Let $\phi\in\cinf_0(\{|x|<\Lambda\})$ have $\phi=1$ on $\supp\chi$, and let $\tilde\phi\in\cinf_0(\{|x|<2\Lambda\})$ have $\tilde\phi=1$ on $\{|x|<\Lambda\}$.

Then $\One_{P_h\ge\Lambda}=\One_{P_h\ge\Lambda}(I-\phi(P_h))$, so
\begin{align*}
    &\left\|(hD)^k\One_{P_h\ge\Lambda}\chi(x,hD)\wh{F}u\right\|_{L^2(\R)}=\left\|(hD)^k\One_{P_h\ge\Lambda}(I-\phi(P_h))\chi(x,hD)\wh{F}u\right\|_{L^2(\R)}\\
    &\le\left\|(hD)^k((I-\phi(P_h))\chi(x,hD)\wh{F}u\right\|_{L^2(\R)}+\left\|(hD)^k\One_{P_h\le\Lambda}(I-\phi(P_h))\chi(x,hD)\wh{F}u\right\|_{L^2(\R)}\\
    &=\left\|(hD)^k(I-\phi(P_h))\chi(x,hD)\wh{F}u\right\|_{L^2(\R)}+\left\|(hD)^k\tilde{\phi}(P_h)(I-\phi(P_h))\chi(x,hD)\wh{F}u\right\|_{L^2(\R)}.
\end{align*}

But by pseudodifferential calculus, $\|(hD)^k(I-\phi(P_h))\chi(x,hD)\|_{L^2\to L^2}$ and $\|(hD)^k\tilde{\phi}(P_h)\|_{L^2\to L^2}$ are both $\Ohi$, and all the other operators are uniformly bounded, so we have for each $k$,
\begin{equation*}
    \left\|(hD)^k\One_{P_h\ge\Lambda}\chi(x,hD)\wh{F}u\right\|_{L^2(\R)}=\Ohi\|u\|_{L^2}
\end{equation*} and hence by Lemma \ref{NormEquivalence}:
\begin{equation*}\label{Lemma4OhiBound}
    \left\|U_h\One_{P_h\ge\Lambda}\chi(x,hD)\wh{F}u\right\|_Y=\Ohi\|u\|_B
\end{equation*}

To handle the second term in (\ref{Lemma4Reduced}), we commute terms and apply Lemmas \ref{ChiYIneq} and \ref{YNormEstimate}:
\begin{align*}
    \left\|U_h\chi(x,hD)\wh{F}u\right\|_Y&=\left\|U_h\chi(x,hD)\One_{\frac{x^2}{\lambda^2}+\lambda^2(hD)^2\le  \Lambda}\wh{F}u\right\|_Y+\Ohi\|u\|_{L^2(\R)}\\
    &=\left\|\widetilde\chi U_h\One_{\frac{x^2}{\lambda^2}+\lambda^2(hD)^2\le  \Lambda}\wh{F}u\right\|_Y+\Ohi\|u\|_{L^2(\R)}\\
    &\le (C+f(m)\sqrt{h})\|U_h\One_{\frac{x^2}{\lambda^2}+\lambda^2(hD)^2\le \Lambda}\wh F u\|_Y+\Ohi\|u\|_{L^2(\R)}\\
    &=(C+f(m)\sqrt{h})\|\wh FU_h\One_{P_h\le\Lambda} u\|_Y+\Ohi\|u\|_{L^2(\R)}\\
    &\le(C\lambda^{-\frac{K}{4}}+f(m)\sqrt{h})\|U_h\One_{P_h\le\Lambda} u\|_Y+\Ohi\|u\|_{L^2(\R)}.
\end{align*}
This gives (\ref{MainEstEq}).
\end{proof}

Finally, we will need pointwise control of $\widetilde{\chi}u$ near $x=0$. We assume from now on that $m\ge K+2$.

\begin{lem}\label{PointwiseLem}
There is an $f:\N\times\R^+\to\R^+$ such that for $k<m-1$, 
\begin{equation}\label{PointwiseLemEq}
    \left|\dd_x^k\widetilde{\chi}u(0)-\dd_x^ku(0)\right|\le f(m,\epsilon)h^{\frac{m-k-1}{2}-\epsilon}\|u\|_Y
\end{equation}
where $\widetilde{\chi}=\chi(\sqrt{h}x,\sqrt{h}D)$ as in Lemma \ref{ChiYIneq}.
\end{lem}
\begin{proof}
We recall the notation form Lemma \ref{ChiYIneq} and write $\tilde h:=\sqrt{h}$ and $\chi_s(x,\xi):=\chi(\tilde h x,\xi)$ so $\widetilde{\chi}=\chi_s(x,\tilde hD)$. 

We first prove the result for $k=0$. Write $\chi_s(x,\tilde hD)=\Op^L_{\tilde h}(\chi_L)+O(\tilde h^{\infty})_{L^2(\R)\to L^2(\R)}$, where $\Op^L$ is the left quantization, and let $\psi(x)=\check\chi_L(0,x)$, where the inverse Fourier transform is in the second variable. We then compute:
 \begin{equation*}
     \widetilde{\chi}u(0)=\int \chi_L(0,\tilde h\xi)\hat{u}(\xi)\,d\xi+O(\tilde h^{\infty})\|u\|_Y.
 \end{equation*}
By Parseval's Identity and Taylor's theorem:
\begin{align}
     \nonumber\int &\chi_L(0,\tilde h\xi)\hat{u}(\xi)\,d\xi
     \nonumber=\tilde h^{-1}\int \psi(\tilde h^{-1}x)u(x)\,dx\\
     \nonumber&=\tilde h^{-1}\int \psi(\tilde h^{-1}x)\left(u(0)+\sum_{j=1}^{p-1}\frac{x^j}{j!}u^{(j)}(0)+\frac{x^p}{(p-1)!}\int_0^1(1-t)^{p-1}u^{(p)}(tx)\,dt\right)\,dx\\
     \label{PointwiseTaylor}&=u(0)+\frac{\tilde h^{-1}}{(p-1)!}\int\psi(\tilde h^{-1}x)x^p\int_0^1(1-t)^{p-1}u^{(p)}(tx)\,dt\,dx
\end{align}
where $p<m$.

We split the integral error term in (\ref{PointwiseTaylor}) as follows. Let:
\begin{equation*}
    I_1:=\tilde h^{-1}\int_{|x|< r}\psi(\tilde h^{-1}x)x^p\int_0^1(1-t)^{p-1}u^{(p)}(tx)\,dt\,dx
\end{equation*}
\begin{equation*}
    I_2:=\tilde h^{-1}\int_{|x|\ge r}\psi(\tilde h^{-1}x)x^p\int_0^{\frac{r}{x}}(1-t)^{p-1}u^{(p)}(tx)\,dt\,dx
\end{equation*}
\begin{equation*}
    I_3:=\tilde h^{-1}\int_{|x|\ge r}\psi(\tilde h^{-1}x)x^p\int_{\frac{r}{x}}^{1}(1-t)^{p-1}u^{(p)}(tx)\,dt\,dx
\end{equation*}
where $0<r<1$ is to be chosen later, so by (\ref{PointwiseTaylor}),
\begin{equation}\label{IntSplit}
    \widetilde{\chi}u(0)-u(0)=\frac{1}{(p-1)!}(I_1+I_2+I_3)+\Ohi\|u\|_Y
\end{equation}

We must bound each $I_j$ in turn. Note that as the result is the same for $\psi(-\cdot), u(-\cdot)$, we may only bound the integrals where $x\ge0$. We then have:
\begin{align}\label{I1}
\begin{split}
    |I_1|&\le\tilde h^{-1}\int_{|x|< r}|\psi(\tilde h^{-1}x)||x|^p\int_0^1|(1-t)^{p-1}u^{(p)}(tx)|\,dt\,dx\\
    &\le r^p\tilde h^{-1}\int_{|x|< r}|\psi(\tilde h^{-1}x)|\,dx\|u^{(p)}\|_{L^{\infty}([-r,r])}\\
    &\le Cr^p\left(\|u^{(p)}\|_{L^2([-1,1])}+\|u^{(p+1)}\|_{L^2([-1,1])}\right)\\
    &\le C(m)r^p\left(\|\la x\ra ^{-m}u^{(p)}\|_{L^2(\R)}+\|\la x\ra ^{-m}u^{(p+1)}\|_{L^2(\R)}\right).
    \end{split}
\end{align}

Also
\begin{align}
\begin{split}\label{I2}
    |I_2|&\le\tilde h^{-1}\int_{|x|\ge r}|\psi(\tilde h^{-1}x)|\int_0^r|x-t|^{p-1}|u^{(p)}(t)|\,dt\,dx\\
    &\le C(p)\tilde h^{-1}\int_{|x|\ge r}|\psi(\tilde h^{-1}x)|\la x\ra^p\,dx\|u^{(p)}\|_{L^{\infty}([-r,r])}\\
    &\le C(p)\tilde h^{-1}\int_{|x|\ge r}|\psi(\tilde h^{-1}x)|\la\tilde h^{-1} x\ra^p\,dx\left(\|\la x\ra ^{-m}u^{(p)}\|_{L^2(\R)}+\|\la x\ra ^{-m}u^{(p+1)}\|_{L^2(\R)}\right)\\
    &=C(p)\int_{|y|\ge \tilde h^{-1}r}|\psi(y)|\la y\ra\,dy\left(\|\la x\ra ^{-m}u^{(p)}\|_{L^2(\R)}+\|\la x\ra ^{-m}u^{(p+1)}\|_{L^2(\R)}\right)\\
    &\le C(m)O(\tilde h^{\infty})\left(\|\la x\ra ^{-m}u^{(p)}\|_{L^2(\R)}+\|\la x\ra ^{-m}u^{(p+1)}\|_{L^2(\R)}\right)
    \end{split}
\end{align}
provided $r>\tilde h^{1-\epsilon}$.

Finally we get 
\begin{align*}
    |I_3|&\le\tilde h^{-1}\int_{|x|\ge r}|\psi(\tilde h^{-1}x)|\int_r^x|x-t|^{p-1}|u^{(p)}(t)|\,dt\,dx\\
    &=\int_{t\ge r}\tilde h^{-1}\int_t^{\infty}|\psi(\tilde h^{-1}x)||x-t|^{p-1}\,dx|u^{(p)}(t)|\,dt.
\end{align*}
But
\begin{align*}
    \tilde h^{-1}\int_t^{\infty}|\psi(\tilde h^{-1}x)||x-t|^{p-1}\,dx&\le \tilde h^{-1}\int_t^{\infty}|\psi(\tilde h^{-1}x)||\tilde h^{-1}(x-t)|^{p-1}\,dx\\
    &=\int_{\tilde h^{-1}t}^{\infty}|\psi(y)||y-\tilde h^{-1}t|^{p-1}\,dy\\
    &=\Psi(\tilde h^{-1} t)
\end{align*}
for $\Psi\in\Sc(\R)$,
so
\begin{align}
\begin{split}\label{I3}
    |I_3|&\le\int_{t\ge r}\Psi(\tilde h^{-1}t)|u^{(p)}(t)|\,dt\\
    &\le\int_{t\ge r}\Psi(\tilde h^{-1}t)\la \tilde h^{-1}t\ra^m\la t\ra^{-m}|u^{(p)}(t)|\,dt\\
    &\le\left(\int_{t>r}\left|\Psi(\tilde h^{-1}t)\right|^2\la \tilde h^{-1}t\ra^{2m}\right)^{\frac{1}{2}}\left(\int_{t>r}\la t\ra^{-2m}\left|u^{(p)}(t)\right|^2\right)^{\frac{1}{2}}\\
    &=C(m)O(\tilde h^{\infty})\|\la x\ra ^{-m}u^{(p)}\|_{L^2(\R)}.
    \end{split}
\end{align}

The bounds (\ref{I1}), (\ref{I2}), (\ref{I3}) together with (\ref{IntSplit}) give (\ref{PointwiseLemEq}) for the case $k=0$, by choosing $p$ as large as possible and $r=\tilde{h}^{1-\epsilon}$.

For larger powers, we use the fact from the proof of Lemma \ref{ChiYIneq}, specifically (\ref{CommWithDer}), which says that
$[D_x^k,\widetilde{\chi}]=\tilde h\chi_k(\tilde hx,\tilde hD)$ where $\chi_k(x,\xi)=0$ near $x,\xi=0$. Then
\begin{equation}\label{DerivComm}
    \left|\dd_x^k\widetilde{\chi}u(0)-\dd_x^ku(0)\right|\le\left|\widetilde{\chi}\dd_x^ku(0)-\dd_x^ku(0)\right|+|\tilde h\chi_k(\tilde hx,\tilde hD)u(0)|,
\end{equation}
and we may simply apply the result for $k=0$ on the first term. On the second term, we write $\chi_k=\chi_k+\phi-\phi$ where $\phi=1$ near $0$ and $\supp\phi\subseteq\chi_k^{-1}(0)$. Then by (\ref{DerivComm}) we get
\begin{align*}
    |\tilde h\chi_k(\tilde hx,\tilde hD)&u(0)|=\tilde h|\chi_k(\tilde hx,\tilde hD)u(0)+\phi(\tilde hx,\tilde hD)u(0)-\phi(\tilde hx,\tilde hD)u(0)|\\
    &\le\tilde h\left|\left(\chi_k(\tilde hx,\tilde hD)+\phi(\tilde hx,\tilde hD)\right)u(0)-u(0)\right|+\left|\phi(\tilde hx,\tilde hD)u(0)-u(0)\right|
\end{align*}
on which we can again apply the result for $k=0$. This proves (\ref{PointwiseLemEq}) in general.
\end{proof}

\subsection{Grushin Problems}
\label{ss:gru}

To determine the spectrum of $\Op_h\chi\wh F$, we use the technique of building a Grushin problem, which relies on the well-known Schur complement formula. We briefly review these techniques here, and refer to Sj\"ostrand and Zworski \cite{SZ} for a more extensive exposition. We also remark that the particular Grushin problem in this section is similar to those in \cite{Gerard}.

The \emph{Schur complement formula} in elementary linear algebra says that for block matrices \begin{equation}\label{SchurMatrices}\mathcal{P}:=\begin{pmatrix}P&R_-\\R_+&0\end{pmatrix},\quad \mathcal{P}^{-1}:=\mathcal{E}:=\begin{pmatrix}E&E_+\\E_-&E_{-+}\end{pmatrix},
\end{equation}
$P$ is invertible if and only if $E_{-+}$ is, and in fact we have the formulas $$P^{-1}=E-E_+E_{-+}^{-1}E_-,\quad E_{-+}^{-1}=-R_+P^{-1}R_-.$$ More generally, the above formula applies even if the operators act on infinite-dimensional spaces. Specifically, let $H_1, H_2, H_+, H_-$ be Hilbert spaces, and let $P:H_1\to H_2$ be an operator. Let $R_-:H_-\to H$, $R_+:H\to H_+$, and define $\mathcal{P}:H_1\oplus H_-\to H_2\oplus H_+$ by (\ref{SchurMatrices}). If $\mathcal{P}$ is invertible, we say that the Grushin problem is \emph{well-posed}. In practice, $H_{\pm}$ are often finite-dimensional or otherwise simpler spaces, and we refer to $E_{-+}$ as the \emph{effective Hamiltonian} of $P$. We remark in particular that if $P$ is a Fredholm operator, it is always possible to construct a Grushin problem with finite-dimensional $H_{\pm}$, which matches the intuition of Fredholm operators being ``effectively finite-dimensional."

The starting point for our computation are the following simple Taylor coefficient operators:
\begin{equation*}
    R_-:\C^{K+1}\to \cinf(\R),\quad R_+:\cinf(\R)\to \C^{K+1}
\end{equation*}
\begin{equation*}
    R_-(a_0,\dots, a_K):=\sum_{j=0}^Ka_j\frac{x^j}{j!},\quad R_+u:=\left(u(0),u'(0),\dots, u^{(K)}(0)\right)
\end{equation*}
Note that $R_+R_-=I_{\C^{K+1}}$ and $R_-R_+$ returns the $K$th order Taylor polynomial at $0$. In particular, this means that $\Ran(I-R_-R_+)\subseteq V_{K+1}$ as defined in (\ref{VKDef}).

We now define the ``Gaussian" versions of $R_{\pm}$ as follows:
\begin{equation*}
    \hat R_-(a_0,\dots, a_K):=e^{-\frac{x^2}{2}}R_-(a_0,\dots, a_K),\quad \hat R_+u:=R_+\left(e^{\frac{x^2}{2}}u\right).
\end{equation*}
The property \begin{equation}\label{RmRp}\hat R_+\hat R_-=I_{\C^{K+1}}\end{equation} still holds, but $\hat{R}_-\hat{R}_+$ is more complicated. We nevertheless still have that $\Ran(I-\hat R_-\hat R_+)\subseteq V_{K+1}$, as for any $u\in\cinf$,
\begin{equation}\label{RpRm}
    (I-\hat{R}_-\hat{R}_+)u=u-e^{-\frac{x^2}{2}}R_-R_+(e^{\frac{x^2}{2}}u)=e^{-\frac{x^2}{2}}(I-R_-R_+)(e^{\frac{x^2}{2}}u)\in V_{K+1}.
\end{equation}
We also note the relation $\hat R_+=B_K R_+$, where $B_K$ is an explicitly computable lower-triangular matrix with all $1$'s on the diagonal.

We are now ready to define the ``localized" $R_{\pm}$ operators. Define
\begin{equation*}
    \tilde R_-:\C^{K+1}\to L^2(\R),\quad \tilde R_+:L^2(\R)\to \C^{K+1}
\end{equation*}
given by
\begin{equation*}
    \tilde{R}_-:=\mathbbm{1}_{P_h\le\Lambda}U_h^*\hat{R}_-,\quad
    \tilde{R}_+:=\hat{R}_+U_h\mathbbm{1}_{P_h\le\Lambda},
\end{equation*} where we note that for $h\le\frac{\Lambda}{2(K+1)}$ the $\mathbbm{1}_{P_h\le\Lambda}$ in the definition of $\tilde{R}_-$ is actually superfluous.
We see that by the Sobolev inequality
\begin{align*}
    (R_+U_h\One_{P_h\le\Lambda}u)_k&=\dd^kU_h\One_{P_h\le\Lambda}u(x)|_{x=0}\\&\ls\left(\|\dd^kU_h\One_{P_h\le\Lambda}u(x)\|_{L^2([-1,1])}+\|\dd^{k+1}U_h\One_{P_h\le\Lambda}u(x)\|_{L^2([-1,1])}\right)\\
    &\ls\|u\|_B
\end{align*}
with constants independent of $h$, so
$\tilde R_+:B\to \C^{K+1}$
is bounded independent of $h$. Similarly
\begin{equation*}
    \|\mathbbm{1}_{P_h\le\Lambda}U_h^*\hat{R}_-v\|_B=\|U_h\mathbbm{1}_{P_h\le\Lambda}U_h^*\hat{R}_-v\|_Y=\|\mathbbm{1}_{P\le \frac{\Lambda}{h}}\hat{R}_-v\|_Y=\|\hat{R}_-v\|_Y
\end{equation*}
for $h$ small enough, which gives in that case that 
$\tilde R_-:\C^{K+1}\to B$
is also bounded independent of $h$.

The analogues of $(\ref{RmRp})$ and $(\ref{RpRm})$ for $\tilde R_{\pm}$ are given in the following proposition.
\begin{prop}\label{LocalGrushins}
If $h\le\frac{\Lambda}{2(K+1)}$, then:
\begin{equation*}
    \tilde R_+\tilde R_-=I_{\C^{K+1}}
\end{equation*}
and $\One_{P_h\le\Lambda}-\tilde R_-\tilde R_+$ has range $V_{K+1}$.
\end{prop}
\begin{proof}
We have that for $h$ small enough:
\begin{align*}
    \tilde R_+\tilde R_-=\hat R_+\One_{P\le \frac{\Lambda}{h}}\hat R_-=I_{\C^{K+1}}
\end{align*}
and also
\begin{align*}
    \tilde R_-\tilde R_+=\mathbbm{1}_{P_h\le\Lambda}U_h^*\hat{R}_-\hat{R}_+U_h\mathbbm{1}_{P_h\le\Lambda}=U_h^*\mathbbm{1}_{P\le \frac{\Lambda}{h}}\hat{R}_-\hat{R}_+\mathbbm{1}_{P\le \frac{\Lambda}{h}}U_h
\end{align*}
hence
\begin{align*}
    \One_{P_h\le\Lambda}-\tilde R_-\tilde R_+&=U_h^*\mathbbm{1}_{P\le \frac{\Lambda}{h}}^2U_h-U_h^*\mathbbm{1}_{P\le \frac{\Lambda}{h}}\hat{R}_-\hat{R}_+\mathbbm{1}_{P\le \frac{\Lambda}{h}}U_h\\
    &=U_h^*\mathbbm{1}_{P\le \frac{\Lambda}{h}}(I-\hat{R}_-\hat{R}_+)\mathbbm{1}_{P\le \frac{\Lambda}{h}}U_h\\
    &=U_h^*(I-\hat{R}_-\hat{R}_+)\mathbbm{1}_{P\le \frac{\Lambda}{h}}U_h
\end{align*}
which has range $V_{K+1}$.
\end{proof}

In the argument to come it will be necessary to apply $\tilde R_+$ to $\chi(x,hD)\wh{F}u.$ The analysis is given by the following lemma, which makes use of Lemma \ref{PointwiseLem} in a fundamental way.
\begin{lem}\label{CommutingRPlus}
For any $u\in L^2(\R)$ we have
\begin{equation}\label{CommutingRPEq}
    \tilde R_+\chi(x,hD)\wh{F}u=L_K\tilde R_+u+Ru
\end{equation}
where $L_K$ is a lower triangular matrix independent of $u$ and $|Ru|\le C_{\epsilon}h^{\frac{m-K-1}{2}-\epsilon}\|u\|_B$.
\end{lem}
\begin{proof}
We compute the following (using analogous arguments to those of Lemma \ref{NormEquivalence}):
\begin{align}
\begin{split}\label{CommutingRPPrep}
    \tilde R_+\chi(x,hD)\wh{F}u&=\hat{R}_+U_h\mathbbm{1}_{P_h\le\Lambda}\chi(x,hD)\wh{F}u\\
    &=\hat{R}_+U_h\chi(x,hD)\wh{F}u+\Ohi\|u\|_{L^2(\R)}\\
    &=\hat{R}_+\wh{F}U_h\chi_{\lambda}(x,hD)u+\Ohi\|u\|_{L^2(\R)}\\
    &=\hat{R}_+\wh{F}U_h\chi_{\lambda}(x,hD)\mathbbm{1}_{P_h\le\Lambda}u+\Ohi\|u\|_{L^2(\R)}\\
    &=\hat{R}_+\wh{F}\widetilde{\chi_{\lambda}}U_h\mathbbm{1}_{P_h\le\Lambda}u+\Ohi\|u\|_{L^2(\R)}\\
    &=B_KD_KR_+\widetilde{\chi_{\lambda}}U_h\mathbbm{1}_{P_h\le\Lambda}u+\Ohi\|u\|_{L^2(\R)}
    \end{split}
\end{align}
 where $D_K=\frac{1}{\sqrt{\lambda}}\mathrm{diag}(1,\frac{1}{\lambda},\dots, \frac{1}{\lambda^K})$.
 
 We must now commute $R_+$ and $\widetilde{\chi_{\lambda}}$. We have by Lemma \ref{PointwiseLem}:
 \begin{align*}
     \dd_x^k\widetilde{\chi_{\lambda}}U_h\mathbbm{1}_{P_h\le\Lambda}u(x)|_{x=0}&=\dd_x^kU_h\mathbbm{1}_{P_h\le\Lambda}u(x)|_{x=0}+O_m(h^{\frac{m-k-1}{2}-\epsilon})\|U_h\One_{P_h\le\Lambda}u\|_Y.
 \end{align*}
 Then
 \begin{equation}\label{CommutingRPFinal}
     B_KD_KR_+\widetilde{\chi_{\lambda}}U_h\mathbbm{1}_{P_h\le\Lambda}u=B_KD_KB_K^{-1}\tilde R_+u+O_m(h^{\frac{m-K-1}{2}-\epsilon})\|u\|_B.
 \end{equation}
Combining (\ref{CommutingRPPrep}) and (\ref{CommutingRPFinal}) gives (\ref{CommutingRPEq}).
\end{proof}

\subsection{Inverting the Grushin Problem}
Let
\begin{equation*}
    \mathcal{P}(z):=\begin{pmatrix}I-z\chi(x,hD)\wh{F}&\tilde R_-\\\tilde R_+&0 \end{pmatrix}.
\end{equation*}

Split $u\in L^2(\R)$ as $u=u_1+u_2$ with $u_1=\One_{P_h\le\Lambda}u$ and $u_2=\One_{P_h\ge\Lambda}u$, and similarly for $v=v_1+v_2$. Consider $\mathcal{P}$ as an operator acting on $\One_{P_h\le\Lambda}L^2(\R)\oplus\One_{P_h\ge\Lambda}L^2(\R)\oplus\C^{K+1}$. Let
\begin{equation*}
    \mathcal{P}(z)\begin{pmatrix}u\\u_-\end{pmatrix}=\begin{pmatrix}v\\v_+\end{pmatrix}
\end{equation*}
so
\begin{equation*}\label{CGru1}
    (I-z\chi\wh{F})(u_1+u_2)+\tilde{R}_-u_-=v_1+v_2
\end{equation*}
\begin{equation*}\label{CGru2}
    \tilde{R}_+(u_1+u_2)=v_+.
\end{equation*}

The equations decouple (up to $\Ohi$ errors) as
\begin{equation*}\label{CGru1}
    (I-z\chi\wh{F}+R_0)u_1+R_1u_2+\tilde{R}_-u_-=v_1
\end{equation*}
\begin{equation*}
    (I+R_3)u_2+R_2u_1=v_2
\end{equation*}
\begin{equation*}\label{CGru2}
    \tilde{R}_+u_1=v_+.
\end{equation*}
where $R_0,R_1, R_2, R_3=\Ohi_{L^2\to L^2}$.

Define the following operator on $L^2(\R)\oplus\C^{K+1}=\One_{P_h\le\Lambda}L^2(\R)\oplus\One_{P_h\ge\Lambda}L^2(\R)\oplus\C^{K+1}$:

\begin{equation}
\label{GrushinAppInverse}
    \tilde{\mathcal{E}}(z):=\begin{pmatrix}\tilde{E}&0&\tilde{E}_+\\0&I&0\\\tilde{E}_-&0&\tilde{E}_{-+} \end{pmatrix}
\end{equation}
and define the nontrivial entries as follows:

\begin{equation*}\label{E}
    \tilde E:=\sum_{j=0}^{\infty}z^j(\chi(x,hD)\wh{F})^j(I-\tilde R_-\tilde R_+)
\end{equation*}
\begin{equation*}\label{EP}
    \tilde E_+:=\tilde R_--\sum_{j=0}^{\infty}z^j(\chi(x,hD)\wh{F})^j(I-\tilde R_-\tilde R_+)(I-z\One_{P_h\le\Lambda}\chi(x,hD)\wh{F})\tilde R_-
\end{equation*}
\begin{equation*}\label{EM}
    \tilde E_-:=\tilde R_+
\end{equation*}
\begin{equation*}\label{EMP}
    \tilde E_{-+}:=-(I_{\C^{K+1}}-zL_K).
\end{equation*}
Recall that by Proposition \ref{LocalGrushins}, $I-\tilde R_-\tilde R_+$ has range $V_{K+1}$, so by Proposition \ref{MainEstimate}, there is a $c>0$ independent of $m$ such that for $h$ small enough 
\begin{equation}\label{DefofE}
\tilde E, \tilde E_+ \text{ defined and bounded for } |z|\le c\lambda^{\frac{K+1}{4}}.
\end{equation}

\begin{lem}\label{Grushin Parametrix}
 For all $m\in\N, K\le m-2$, there is an $r_K$  with $r_K\to\infty$ such that for $|z|\le r_K$, $h$ small enough, the operator $\tilde{\mathcal{E}}(z)$ is well-defined and satisfies
\begin{equation*}
    \mathcal{P}(z)\tilde{\mathcal{E}}(z)=I+\mathcal{T}(z)
\end{equation*}
with 
\begin{equation}\label{TBound}\|\mathcal{T}(z)\|_{B\oplus\C^{K+1}\to B\oplus\C^{K+1}}=O(h^{\frac{m-K-1}{4}-\epsilon}).\end{equation}

\end{lem}
\begin{proof}
Let $r_K=c\lambda^{\frac{K}{8}}$, where $c$ is given in (\ref{DefofE}).

Let
\begin{equation*}\label{TDecay}
    \mathcal{T}(z)=\begin{pmatrix}T_{11}&T_{12}&T_{1+}\\
    T_{21}&T_{22}&T_{2+}\\T_{+1}&T_{+2}&T_{++}\end{pmatrix}
\end{equation*}
on $\One_{P_h\le\Lambda}L^2(\R)\oplus\One_{P_h\ge\Lambda}L^2(\R)\oplus\C^{K+1}$, so expanded out we have:
\begin{equation}\label{TDef}
    \begin{pmatrix}I+T_{11}&T_{12}&T_{1+}\\
    T_{21}&I+T_{22}&T_{2+}\\T_{+1}&T_{+2}&I+T_{++}\end{pmatrix}=\\\begin{pmatrix}I-z\chi(x,hD)\wh{F}+R_0&R_1&\tilde{R}_-\\R_2&I+R_3&0\\\tilde{R}_+&0&0 \end{pmatrix}\begin{pmatrix}\tilde{E}&0&\tilde{E}_+\\0&I&0\\\tilde{E}_-&0&\tilde{E}_{-+} \end{pmatrix}
\end{equation}
It remains to compute the product entries in (\ref{TDef}).
We have by telescoping that:
\begin{align*}
\begin{split}\label{T11}
    I+T_{11}&=(I-z\chi(x,hD)\wh F)\left(\sum_{j=0}^{\infty}z^j(\chi(x,hD)\wh{F})^j(I-\tilde R_-\tilde R_+)\right)+\tilde R_-\tilde R_++R_0\tilde{E}\\
    &=I-\tilde R_-\tilde R_++\tilde R_-\tilde R_++R_0\tilde E\\
    &=I+\Ohi.
    \end{split}
\end{align*}
The equation
\begin{equation*}\label{T12}
    T_{12}=R_1=\Ohi
\end{equation*}
is trivial, and for $T_{1+}$ we compute by Lemma \ref{CommutingRPlus}:
\begin{align*}
\begin{split}\label{T1+}
    T_{1+}&=(I-z\chi(x,hD)\wh F)\tilde R_--(I-\tilde R_-\tilde R_+)(I-z\One_{P_h\le\Lambda}\chi(x,hD)\wh{F})\tilde R_-\\
    &\quad+\tilde R_-(-I_{\C^{K+1}}+zL_K)+R_0\tilde E_+\\
    &=(I-z\chi(x,hD)\wh F)\tilde R_--(I-\tilde R_-\tilde R_+)(I-z\chi(x,hD)\wh{F})\tilde R_-+\tilde R_-(-I_{\C^{K+1}}+zL_K)\\
    &\quad+(I-\tilde R_-\tilde R_+)(I-z\One_{P_h\ge\Lambda}\chi(x,hD)\wh{F})\tilde R_-+R_0\tilde E_+\\
    &=\tilde R_-\tilde R_+(I-z\chi(x,hD)\wh{F})\tilde R_-+\tilde R_-(-I_{\C^{K+1}}+zL_K)\\
    &\quad-(I-\tilde R_-\tilde R_+)(z\One_{P_h\ge\Lambda}\chi(x,hD)\wh{F})\tilde R_-+R_0\tilde E_+\\
   &=\tilde R_-R\tilde R_--(I-\tilde R_-\tilde R_+)(z\One_{P_h\ge\Lambda}\chi(x,hD)\wh{F})\tilde R_-+R_0\tilde E_+\\
   &=\Ohi.\end{split}
\end{align*}
The next three
\begin{equation*}\label{T2*}
    T_{21}=R_2\tilde E,\quad I+T_{22}=I+R_3,\quad T_{2+}=R_2\tilde E_+
\end{equation*} are trivially all equal to $\Ohi$, and $T_{+2}=0$.

The bounds on $T_{+1}$ and $T_{+3}$ require a more delicate argument. Fix $J$ to determined later. Then we have:
\begin{align*}
    T_{+1}&=\tilde R_+\sum_{j=0}^{\infty}z^j(\chi(x,hD)\wh{F})^j(I-\tilde R_-\tilde R_+)\\
    &=\sum_{j=0}^{J}z^j\tilde R_+(\chi(x,hD)\wh{F})^j(I-\tilde R_-\tilde R_+)+\tilde R_+\sum_{j=J+1}^{\infty}z^j(\chi(x,hD)\wh{F})^j(I-\tilde R_-\tilde R_+).
\end{align*}
By Lemma \ref{MainEstimate}, for $|z|\le r_k$
\begin{align}
\begin{split}\label{LateTerms1}
    &\left\|\tilde R_+\sum_{j=J+1}^{\infty}z^j(\chi(x,hD)\wh{F})^j(I-\tilde R_-\tilde R_+)\right\|_{B\to\C^{K+1}}\\&\le C\sum_{j=J+1}^{\infty}z^j\|(\chi(x,hD)\wh{F})^j(I-\tilde R_-\tilde R_+)\|_{B\to B}\le C\sum_{j=J+1}^{\infty}\lambda^{-\frac{K}{8}}\le C\lambda^{-\frac{JK}{8}}.
    \end{split}
\end{align}
Meanwhile repeated application of Lemma \ref{CommutingRPlus} gives
\begin{equation*}
    R_+(\chi(x,hD)\wh{F})^j=L^j\tilde R_++\sum_{i=0}^{j-1}L^iR(\chi(x,hD)\hat{F})^{j-1-i}
\end{equation*}
with
\begin{equation*}
    \left\|\sum_{i=0}^{j-1}L^iR(\chi(x,hD)\hat{F})^{j-1-i}\right\|_{B\to\C^{K+1}}\le CJh^s
\end{equation*}
for $s=\frac{m-K-1}{2}-\epsilon$ and $C=C(\epsilon)$. Then
\begin{equation}\label{EarlyTerms1}
    \left\|\sum_{j=0}^{J}z^j\tilde R_+(\chi(x,hD)\wh{F})^j(I-\tilde R_-\tilde R_+)\right\|_{B\to\C^K}\le CJr_K^Jh^s=CJ\lambda^{\frac{JK}{8}}h^s.
\end{equation}
Combining (\ref{LateTerms1}) and (\ref{EarlyTerms1}) gives:
\begin{equation*}
    \|T_{+1}\|_{B\to\C^K}\le C\left(\lambda^{-\frac{JK}{8}}+J\lambda^{\frac{JK}{8}}h^s\right).
\end{equation*}
To optimize, set $J=\frac{4s}{K\log \lambda}\log\frac{1}{h}$, which gives the bound
\begin{equation}\label{TPlus1Bound}
    \|T_{+1}\|_{B\to\C^{K+1}}\le Ch^{\frac{s}{2}}\log\frac{1}{h}.
\end{equation}

Finally to bound $T_{++}$ we use (\ref{TPlus1Bound}) to get:
\begin{align*}
\begin{split}\label{T++}
    I+T_{++}&=\tilde R_+\tilde R_--\tilde R_+\sum_{j=0}^{\infty}z^j(\chi(x,hD)\wh{F})^j(I-\tilde R_-\tilde R_+)(I-z\One_{P_h\le\Lambda}\chi(x,hD)\wh{F})\tilde R_-\\
    &=I-T_{+1}\tilde R_-\\
    &=I+O\left(h^{\frac{s}{2}}\log\frac{1}{h}\right)_{\C^{K+1}\to \C^{K+1}}.
    \end{split}
\end{align*}
By choosing a new $\epsilon$, this proves the bound (\ref{TBound}) for $h$ small enough, and completes the proof of the lemma.
\end{proof}

\begin{cor}\label{grushininverse}
For all $m\in\N, K\le m-2$, there is an $r_K$  with $r_K\to\infty$, such that for $h$ small enough there is a $\mathcal{E}(z):L^2(\R)\oplus \C^{K+1}\to L^2(\R)\oplus \C^{K+1}$ defined on $\{|z|\le R_K\}$ such that
\begin{equation*}
\mathcal{P}(z)\mathcal{E}(z)=\mathcal{E}(z)\mathcal{P}(z)=I.
\end{equation*}
\end{cor}
\begin{proof}
We first find a right inverse via Neumann series. For $h$ small enough we may define
\begin{equation}\label{NeumannDef}
    \mathcal{E}(z):=\tilde{\mathcal{E}}(z)\left(\sum_{j=0}^{\infty}(-1)^j\mathcal{T}(z)^j\right)
\end{equation}
which gives that $\mathcal{P}(z)\mathcal{E}(z)=I.$

We now claim that $\mathcal{P}(z)$ is an analytic family of Fredholm operators of index $0$. For indeed for $z=0$ we have
\begin{equation*}
    \mathcal{P}(0)=\begin{pmatrix}I&\tilde R_-\\\tilde R_+&0 \end{pmatrix}
\end{equation*}
and one may easily check that
\begin{equation*}
\mathcal{P}(0)^{-1}=\begin{pmatrix}I-\tilde R_-\tilde R_+&\tilde R_-\\\tilde R_+&-I \end{pmatrix}
\end{equation*}
is a two-sided inverse. Then by continuity $\mathcal{P}(z)$ has index zero for all $z$.

This shows that the right inverse $\mathcal{E}(z)$ defined in (\ref{NeumannDef}) is a two-sided inverse, as desired.
\end{proof}

We may now estimate the spectrum of $\chi\wh{F}$ using the Schur complement formula. We have:
\begin{equation*}
    \left(\mathcal{T}^j\right)_{+-}=O(h^{pj})
\end{equation*}
for $p=\frac{m-K-1}{4}-\epsilon$. This gives that
\begin{equation*}
    E_{-+}=-I+zL_K+O(h^p)
    =B_K(-I+zD_K+O(h^p))B_K^{-1},
\end{equation*}
which shows that (using boundedness of all operators involved to bound $\frac{1}{z}$ below) $\frac{1}{z}\in\sigma(\chi\wh{F})$ if and only if $D_K-\frac{1}{z}I=O(h^p)$. This implies $\frac{1}{z}\in\sigma_{h^p}(D_K)$ (where $\sigma_\epsilon$ denotes the $\epsilon$-pseudospectrum), meaning by self-adjointness that $\frac{1}{z}=\lambda^{\frac{2k+1}{2}}+O(h^p)$.

\begin{proof}[Proof of Theorem \ref{mainL2}]

We generalize the above to metaplectic operators on $L^2(\R)$. Let $M\in Sp(2,\R)$ have eigenvalues $\lambda,\lambda^{-1}$ with $\lambda>1$. Then $M=Q^{-1}FQ$, where $F=\begin{pmatrix}\lambda&0\\0&\lambda^{-1}\end{pmatrix}$ is as before, so 
\begin{equation}
\label{Qconj}
   \Op_h\chi\wh{M}=\Op_h\chi\wh{Q^{-1}}\wh{F}\wh{Q}
   =\wh{Q^{-1}}\Op_h(\chi\circ Q^{-1})\wh{F}\wh{Q}.
\end{equation}
and hence has the same spectrum as $\Op_h(\chi\circ Q^{-1})\wh{F}$.

Thus by picking $K$ such that $\frac{1}{r_K}<\delta$ and $m\gg K$ as large as we want, we have shown Theorem \ref{mainL2}.
\end{proof}

In Section \ref{s:torus} of the paper, we will need an actual Grushin problem for the general metaplectic operator. But replacing $\mathcal{P}(z)$ with $\mathcal{P}^Q(z)=\mathcal{Q}^{-1}\mathcal{P}(z)\mathcal{Q}$ where $\mathcal{Q}=\begin{pmatrix}\wh Q&0\\0&I\end{pmatrix}$, we have our new block matrix
\begin{equation}\label{PQ}
\mathcal{P}^Q(z)=\begin{pmatrix} I-z\Op_h(\chi_Q)\wh{M}&\wh Q^{-1}\tilde{R}_-\\\tilde{R}_+\wh Q&0\end{pmatrix}
\end{equation}
where $\chi_Q=\chi\circ Q^{-1}$.
Then replacing $\mathcal{E}(z)$ with $\mathcal{Q}^{-1}\mathcal{E}(z)\mathcal{Q}$ gives the desired well-posedness.

\section{The operator on the torus}
\label{s:torus}

\subsection{Quantization of the Torus} 
\label{ss:microlocalHN}

We describe the space of quantum states on $\T^2$, as seen in \cite{DG}. Let $$H_N:=\mathrm{span}\left\{\frac{1}{\sqrt{N}}\sum_{k\in\Z}\delta_{x=k+\frac{n}{N}}:n\in\{0,1,\dots,N-1\}\right\}$$ with $h=\frac{1}{2\pi N}$. These are precisely the elements of $\Sc$ which are periodic in physical and phase space under the semiclassical Fourier transform $\mathcal{F}_h$. For simplicity, define $Q_n=\frac{1}{\sqrt{N}}\sum_{k\in\Z}\delta_{x=k+\frac{n}{N}}$, which we decide by definition to be an orthonormal basis of $H_N$. One may remark that this is an elementary example of the general geometric Toeplitz quantization studied in \cite{Deleporte}.

We review the pseudodifferential calculus on $H_N$, which also may be found in \cite{CZ},
\cite{DJ}, \cite{Schenck}. Let $a\in \cinf(\T^2)$. Then $a$ lifts to a doubly-periodic function $\tilde{a}$ on $T^*\R$, which is seen to be in the symbol class $S(1)$. Define $$\Op_Na=\Op_ha|_{H_N}$$ which is a map from $H_N$ to itself that is also given in coordinates by $$\Op_N(a)Q_j=\sum_{m=0}^{N-1}A_{mj}Q_m$$ with $$A_{mj}=\sum_{k,l\in\Z}\hat{a}(k,j-m-lN)(-1)^{kl}e^{\pi i\frac{(j+m)k}{N}}.$$

Most results from pseudodifferential calculus carry over to $H_N$, see \cite{CZ},
\cite{DJ} for details. In particular, there is an analogue of Lemma \ref{RnFunCalc}, given as follows.

\begin{lem}\label{TnFunCalc}
If $a\in S(1)$ on $\T^2$, and $\phi\in\cinf_0(\T^2)$, then $\phi(\Op_Na)$ is a Weyl pseudodiffernetial operator on $H_N$. Suppose $b\in S(1)$ on $\T^2$ such that $\supp b\cap a^{-1}(\supp\phi)=\emptyset$. Then $\phi(\Op_N a)\Op_Nb=\Ohi_{H_N\to H_N}$.
\end{lem}

The proof of Lemma \ref{TnFunCalc} essentially consists of showing that the action of $\phi(\Op_Na)$ on $H_N$ is the same as the action of the matrix $\phi$ applied to the matrix with entries $\la \phi(\Op_Na)Q_j,Q_k\ra$. The details are given (in a slightly more general situation) as Lemma 2.8 in \cite{CZ}.

In addition, a particular class of metaplectic operators can be defined on $H_N$. Let $M\in Sp(2n,\Z)$ be an integer symplectic matrix, and let $N$ be even. Define $\wh M_N=\wh M|_{H_N}$, which is a unitary map on $H_N$. Consequently, we have an analogous exact Egorov's theorem:
\begin{equation*}
    \Op_N(a)\wh{M}_N=\wh{M}_N\Op_N(a\circ M).
\end{equation*}

\subsection{Transplantation Operators}
\label{ss:transplants}

We consider the following ``sampling" maps between the torus and the real line. Let $\psi\in\C_0^{\infty}(\T^*\R)$ be a bump function at $(0,0)$. The sampling operator $S_{\psi}:L^2(\R)\to H_N$ is defined by $$S_{\psi}u(x)=N^{-\frac{1}{2}}\sum_{n=0}^{N-1}a_nQ_n$$ where $$a_n=\sum_{k\in\Z}(\Op_h(\psi) u)(k+\frac{n}{N})$$ and as usual $h=\frac{1}{2\pi N}$.
We may also write $S_{\psi}=\Pi_N\Op_h\psi$, where $\Pi_N:\Sc(\R)\to H_N$  given by $$\Pi_Nu=N^{-\frac{1}{2}}\sum_{n=0}^{N-1}\sum_{k\in\Z}u(k+\frac{n}{N})Q_n$$ are the projection operators considered in \cite{DJ}.
As $\Op_h(\psi):L^2(\R)\to\Sc(\R)$, we easily see that $S_{\psi}$ is well defined and bounded. In fact, we shall soon show that it is bounded independent of $h$.

The benefit of working with $S_{\psi}$ comes from the particularly nice form of $S_{\psi}^*$ and $S_{\psi}S_{\psi}^*$.

\begin{prop}\label{SamplingAdjoint}
The adjoint of $S_{\psi}$ is $\Op_h\overline{\psi}$, acting on $H_N$ as a subset of $\Sc'(\R)$.
\end{prop}
\begin{proof}
Let $K(x,y)$ be the Schwartz kernel of $\Op_h\psi$. Then for $u\in L^2(\R)$, $v=\sum_{n=0}^{N-1}Q_n$, we have
\begin{equation*}
    \la S_{\psi}u,v\ra_{H_N}=N^{-\frac{1}{2}}\sum_{n=0}^{N-1}\overline{v_n}\sum_{k\in\Z}(\Op_h(\psi)\,u)(k+\frac{n}{N})
    N^{-\frac{1}{2}}\int\sum_{n=0}^{N-1}\overline{v_n}\sum_{k\in\Z}K(k+\frac{n}{N},y)u(y)\,dy
\end{equation*}
which gives
\begin{equation*}
    S_{\psi}^*v(y)=N^{-\frac{1}{2}}\sum_{n=0}^{N-1}\sum_{k\in\Z}\overline{K(k+\frac{n}{N},y)}v_n\\
    =(\Op_h(\overline{\psi})v)(y)
\end{equation*}
as desired.
\end{proof}

\begin{prop}\label{ProjPsiDo}
If $\psi\in\cinf_0(\T^2)\subseteq\cinf_0(\R^2)$, then
\begin{equation}\Pi_N\Op_h\psi|_{H_N}=\Op_N\tilde{\psi},\end{equation}
where $\tilde{\psi}(x,\xi)=\sum_{k,l\in\Z}\psi(x+k,\xi+l)$ is the periodization of $\psi$.
\end{prop}
\begin{proof}
Let $u\in H_N$.
By periodicity of $u$, we have formally:
\begin{align*}
    \Pi_N\Op_h\psi u&=N^{-\frac{1}{2}}\sum_{n=0}^{N-1}\sum_{k\in\Z}\frac{1}{2\pi h}\iint e^{i(k+\frac{n}{N}-y)\xi}\psi\left(\frac{k+\frac{n}{N}+y}{2},\xi\right)u(y)\,dy\,d\xi Q_n\\
    &=N^{-\frac{1}{2}}\sum_{n=0}^{N-1}\sum_{k\in\Z}\frac{1}{2\pi h}\iint e^{i(\frac{n}{N}-y)\xi}\psi\left(\frac{\frac{n}{N}+y}{2}+k,\xi\right)u(y+k)\,dy\,d\xi Q_n\\
    &=N^{-\frac{1}{2}}\sum_{n=0}^{N-1}\frac{1}{2\pi h}\iint e^{i(\frac{n}{N}-y)\xi}\sum_{k\in\Z}\psi\left(\frac{\frac{n}{N}+y}{2}+k,\xi\right)u(y)\,dy\,d\xi Q_n\\
    &=N^{-\frac{1}{2}}\sum_{n=0}^{N-1}(\Op_h(\psi_p)u) Q_n
\end{align*}
where $\psi_p(x,\xi)=\sum_{k\in\Z}\psi(x+k,\xi)$.

Now let $\psi_L=e^{\frac{i}{2}h\la D_x,D_{\xi}\ra}\psi_p,$ and remark that $\psi_L$ is also periodic in the $x$ variable. Then we get the left quantization
\begin{align*}
    N^{-\frac{1}{2}}\sum_{n=0}^{N-1}(\Op_h^L\psi_Lu) Q_n&=N^{-1}\sum_{k\in\Z}\sum_{n=0}^{N-1}\delta_{x=k+\frac{n}{N}}\frac{1}{2\pi h}\int e^{\frac{i}{h}x\xi}\chi_L(x,\xi)\hat u(\xi)\,d\xi\\
    &=\sum_{m\in\Z}\delta_{x=\frac{m}{n}}\int e^{\frac{i}{h}x\xi}\chi_L(x,\xi)\hat u(\xi)\,d\xi\\
    &=N\sum_{m\in\Z}e^{2\pi iNmx}\int e^{\frac{i}{h}x\xi}\chi_L(x,\xi)\hat u(\xi)\,d\xi\\
    &=\frac{1}{2\pi h}\int\sum_{m\in\Z}e^{\frac{i}{h}x(\xi+m)}\chi_L(x,\xi)\hat u(\xi)\,d\xi\\
    &=\frac{1}{2\pi h}\int\sum_{m\in\Z}e^{\frac{i}{h}x\xi}\sum_{m\in\Z}\chi_L(x,\xi)\hat u(\xi+m)\,d\xi,
\end{align*}
where the third line follows from Poisson summation formula and the final line from periodicity of $\hat{u}$.
But as $e^{\frac{i}{2}h\la D_x,D_{\xi}\ra}$ commutes with periodization, we have this is just $\Op_h(\tilde\chi) u$, as desired.
\end{proof}

\begin{cor}\label{HNHN}
\begin{equation}\label{HNHNEq}
S_{\psi}S_{\psi^*}=\Op_N\left(\widetilde{\psi\#\overline\psi}\right),\end{equation} and consequently $S_{\psi}:L^2(\R)\to H_N$ is bounded independently of $h$. 
\end{cor}
\begin{proof}
We simply compute:
\begin{equation*}S_{\psi}S_{\psi}^*=\Pi_N\Op_h\psi\Op_h\overline\psi|_{H_N}=\Pi_N\Op_h\left(\widetilde{\psi\#\overline\psi}\right)|_{H_N}=\Op_N\left(\widetilde{\psi\#\overline\psi}\right).
\end{equation*}
which gives (\ref{HNHNEq}).
\end{proof}

We now construct ``spectral cutoffs" which will allow us to move our result about $L^2$ onto $H_N$. Let $f\in\cinf(\T^2)$ have $f(x,\xi)=(x^2+\xi^2)$ on a neighborhood of $(0,0)$ that will be specified later, and $f(x,\xi)=1$ outside a slightly larger neighborhood. In addition, we assume $f$ to be nondecreasing with $|x|^2+|\xi|^2$. Let $v_1,\dots v_B\in H_N$ be eigenfunctions of $\Op_NF$ corresponding to the smallest $B$ eigenvalues $\lambda_1,\dots,\lambda_B$, such that $\lambda_B<\Lambda', \lambda_{B+1}>\Lambda'+\Omega(h^k)$ for some $k$, and for $\Lambda'>\Lambda$ to be determined later. We do suppose that $\psi=1$ on $\{|x|^2+|\xi|^2\le \Lambda'
\}$. (By Weyl's law, such an $\Lambda'$, $B$ exist in any $O(h)$-neighborhood.) Let $H_{\Lambda'}=\spn\{v_1,\dots v_B\}$. These eigenfunctions serve intuitively as ``$H_N$-analogues" to the familiar Hermite functions, and we shall use this analogy in constructing our local Grushin problem. Refer to Figure \ref{f:nestedcutoffs} for a diagram of the functions $\chi, \psi$, and $f$ in relation to the balls of radii $\sqrt{\Lambda}, \sqrt{\Lambda'}$.

Let $T_0=S_{\psi}^*\One_{F\le \Lambda'}=S_{\psi}^*\Pi_{H_{\Lambda'}}$. A key fact is that $T_0$ is almost unitary on $H_{\Lambda'}$.



\begin{lem}\label{AlmostUnitarity}
Let $f=x^2+\xi^2$ on $\supp\psi$. Then 
\begin{equation*}
T_0^*T_0=\Pi_{H_{\Lambda'}}(I+R_{\Lambda'})
\end{equation*}
where $R_{\Lambda'}|_{H_{\Lambda'}^{\perp}}=0$ and $\|R_{\Lambda'}\|_{H_N\to H_N}=\Ohi$.
\end{lem}
\begin{proof}
We trivially see that for $w\in H_{\Lambda'}^{\perp}$, $T_0w=0$. Let $v\in H_{\Lambda'}$. Then for $\psi\in\cinf(\T^2)$ with $\phi=1$ on $|x|^2+|\xi|^2\le E$ and $\tilde\psi=1$ on $\supp\phi$, we get by Lemma \ref{RnFunCalc} that
\begin{align*}
    T_0^*T_0v&=\Pi_{H_{\Lambda'}}S_{\psi}S_{\psi}^*\Pi_{H_{\Lambda'}}v=\Pi_{H_{\Lambda'}}\Op_N(\psi\sharp\overline\psi)\One_{F\le \Lambda'}v=\Pi_{H_{\Lambda'}}\Op_N(\psi\sharp\overline\psi)\phi(F)\One_{F\le \Lambda'}v\\
    &=\Pi_{H_{\Lambda'}}\phi(F)\One_{F\le \Lambda'}v+\Pi_{H_{\Lambda'}}\Ohi v=\Pi_{H_{\Lambda'}}(I+R_{\Lambda'})v
\end{align*}
as desired.
\end{proof}

We now show that for any eigenvector $v$ of $F$ with eigenvalue at most $\Lambda'$, $T_0v$ is very close to being an eigenvector of $P_h$ with the same eigenvalue.
\begin{lem}\label{AlmostEig}
If $Fv=\lambda v$ with $\lambda\le \Lambda'$, then $P_hT_0v=\lambda v+w$ with $\|w\|_{L^2(\R)}=\Ohi.$
\end{lem}
\begin{proof}
If $Fv=\lambda v$, then letting $\phi\in\cinf(\T)$ be as in Lemma \ref{AlmostUnitarity} gives
\begin{align*}
    P_hS_{\psi}^*\One_{F\le \Lambda'}v&=P_hS_{\psi}^*\phi(F)v=S_{\psi}^*\phi(F)Fv+\Ohi=\lambda S_{\psi}^*\One_{F\le \Lambda'}v+\Ohi.
\end{align*}
from which the result follows.
\end{proof}

We are now ready to construct our spectral cutoffs. Recall from Lemma \ref{AlmostUnitarity} that $T_0^*T_0=\Pi_{H_{\Lambda'}}(I+R_{\Lambda'})$. By the $\Ohi$ bound on $R_{\Lambda'}$, for $h$ small enough $(I+R_{\Lambda'})^{-\frac{1}{2}}$ is a well-defined bounded self-adjoint operator, so let $T=T_0(I+R_{\Lambda'})^{-\frac{1}{2}}$. Then $T$ is a unitary transformation on $H_{\Lambda'}$, and by Lemma \ref{AlmostEig} and eigenvalue separation near $\Lambda'$, one has $T:H_{\Lambda'}\to\One_{P_h\le\Lambda}L^2(\R)$.

The other key property of $T$ is that it plays well with our cutoff metaplectic operator.

\begin{prop}
Given metaplectic $\wh{M}$ with $M\in Sp(2,\Z)$, $$T^*(I-z\chi(x,hD)\wh{M})T=\Pi_{\Ran T^*}-z\Op_N\chi\wh{M}_N+\Ohi_{H_N\to H_N}.$$
\end{prop}
\begin{proof}
We must only compute $T^*\chi(x,hD)\wh{M}T$. But this is
\begin{align*}
\One_{F\le \Lambda'}S_{\psi}\chi(x,hD)\wh{M}S_{\psi}^*\One_{F\le \Lambda'}&=\One_{F\le \Lambda'}S_{\psi}\chi(x,hD)\Op_N(\psi\circ M^{-1})\wh{M}_N\One_{F\le \Lambda'}+\Ohi\\
&=\One_{F\le \Lambda'}\Op_N\chi\wh{M}_N\One_{F\le \Lambda'}+\Ohi.
\end{align*}
Let $\tilde\phi$ have $\supp\tilde\phi\subseteq\{|x|^2+|\xi|^2<\Lambda'\}$, with $\tilde\phi=1$ on $\supp\chi\cup\supp\chi_{\lambda}$.
Then
\begin{align*}
    \One_{F\le \Lambda'}\Op_N\chi\wh{M}_N\One_{F\le \Lambda'}&=\One_{F\le \Lambda'}\tilde\phi(F)\Op_N\chi\wh{M}_N\tilde\phi(F)\One_{F\le \Lambda'}+\Ohi\\
    &=\tilde\phi(F)\Op_N\chi\wh{M}_N\tilde\phi(F)+\Ohi\\
    &=\Op_N\chi\wh{M}_N+\Ohi
\end{align*}
which completes the proof.
\end{proof}

\subsection{Grushin Problem on the Torus}
\label{ss:gruT}
Let 
\begin{equation}\label{GrushinOpTorus}\mathcal{P}_N(z):=\begin{pmatrix}I-z\Op_N\chi_Q\wh{M}_N&T^*\wh{Q}^{-1}\tilde R_-\\\tilde R_+\wh{Q}T&0 \end{pmatrix}\end{equation} as a map from $H_N\oplus\C^{K+1}\to H_N\oplus\C^{K+1}$.
We need one more lemma about different sized spectral cutoffs on $L^2(\R).$

\begin{lem}\label{twospeccutoffs}
If $\Lambda'>\|Q\|^2\Lambda$, then 
\begin{equation*}\One_{P_h\le\Lambda}\wh{Q}^{-1}\One_{P_h\le\Lambda'}=\One_{P_h\le\Lambda}\wh{Q}^{-1}+\Ohi_{L^2(\R)\to L^2(\R)}.
\end{equation*}
\end{lem}
\begin{proof}
Let $\phi_1,\phi_2\in\cinf(T^*\R)$ with the following properties:
\begin{itemize}
    \item $\supp\phi_2\subseteq|x|^2+|\xi|^2<\Lambda'$.
    \item $\phi_2(x,\xi)=1$ whenever $(\|Q\|^{-1}x,\|Q\|^{-1}\xi)\in\supp\phi_1$.
    \item $\phi_1(x,\xi)=1$ whenever $|x|^2+|\xi|^2\le \Lambda$.
\end{itemize}
Let $\phi_1(P_h)=\Op_h\phi_{P1}$, $\phi_2(P_h)=\Op_h\phi_{P2}$.
Then by repeated applications of Lemma \ref{RnFunCalc} and exact Egorov theorem:
\begin{align*}
    \One_{P_h\le\Lambda}\wh{Q}^{-1}\One_{P_h\le\Lambda'}&=\One_{P_h\le\Lambda}\phi_1(P_h)\wh{Q}^{-1}\One_{P_h\le\Lambda'}\\
    &=\One_{P_h\le\Lambda}\wh{Q}^{-1}\Op_h(\phi_{P1}\circ Q)\One_{P_h\le\Lambda'}\\
    &=\One_{P_h\le\Lambda}\wh{Q}^{-1}\Op_h(\phi_{P1}\circ Q)\Op_h\phi_{P2}\One_{P_h\le\Lambda'}+\Ohi_{L^2(\R)\to L^2(\R)}\\
    &=\One_{P_h\le\Lambda}\wh{Q}^{-1}\Op_h(\phi_{P1}\circ Q)\phi_2(P_h)\One_{P_h\le\Lambda'}+\Ohi_{L^2(\R)\to L^2(\R)}\\
    &=\One_{P_h\le\Lambda}\wh{Q}^{-1}\Op_h(\phi_{P1}\circ Q)\Op_h\phi_{P2}+\Ohi_{L^2(\R)\to L^2(\R)}\\
    &=\One_{P_h\le\Lambda}\Op_h(\phi_{P1})\wh{Q}^{-1}+\Ohi_{L^2(\R)\to L^2(\R)}\\
    &=\One_{P_h\le\Lambda}\wh{Q}^{-1}+\Ohi_{L^2(\R)\to L^2(\R)}
\end{align*}
which completes the proof.
\end{proof}

Split $u\in H_N$ as $u=u_1+u_2$ with $u_1=\One_{F\le \Lambda'}u$ and $u_2=\One_{F\ge \Lambda'}u$, and similarly for $v=v_1+v_2$.

Consider $\mathcal{P}_N$ as an operator acting on $H_{\Lambda'}\oplus H_{\Lambda'}^{\perp}\oplus\C^{K+1}$. Let
\begin{equation*}
    \mathcal{P}_N(z)\begin{pmatrix}u\\u_-\end{pmatrix}=\begin{pmatrix}v\\v_+\end{pmatrix}
\end{equation*}
so
\begin{equation*}
    (I-z\Op_N\chi_Q\wh{M}_N)(u_1+u_2)+T^*\wh{Q}^{-1}\tilde{R}_-u_-=v_1+v_2
\end{equation*}
\begin{equation*}
    \tilde{R}_+\wh QT(u_1+u_2)=v_+.
\end{equation*}

As before, the equations decouple (up to $\Ohi$ errors) as
\begin{equation*}
    (I-z\Op_n\chi_Q\wh{M}_N+S_0)u_1+S_1u_2+T^*\wh{Q}^{-1}\tilde{R}_-u_-=v_1
\end{equation*}
\begin{equation*}
    (I+S_3)u_2+S_2u_1=v_2
\end{equation*}
\begin{equation*}
    \tilde{R}_+\wh{Q}Tu_1=v_+.
\end{equation*}
where $R_1, R_2=\Ohi_{L^2\to L^2}$. Then we can write
\begin{equation}\label{grushinT}
    \mathcal{P}_N(z):=\begin{pmatrix}I-z\Op_N\chi_Q\wh{M}_N+S_0&S_1&T^*\wh{Q}^{-1}\tilde R_-\\S_2&I+S_3&0\\\tilde R_+\wh QT&0&0 \end{pmatrix}
\end{equation}
as an operator on $\One_{F\le \Lambda'}H_N\oplus\One_{Q> \Lambda'}H_N\oplus\C^{N+1}$, with $S_0, S_1, S_2, S_3=\Ohi_{H_N\to H_N}$.

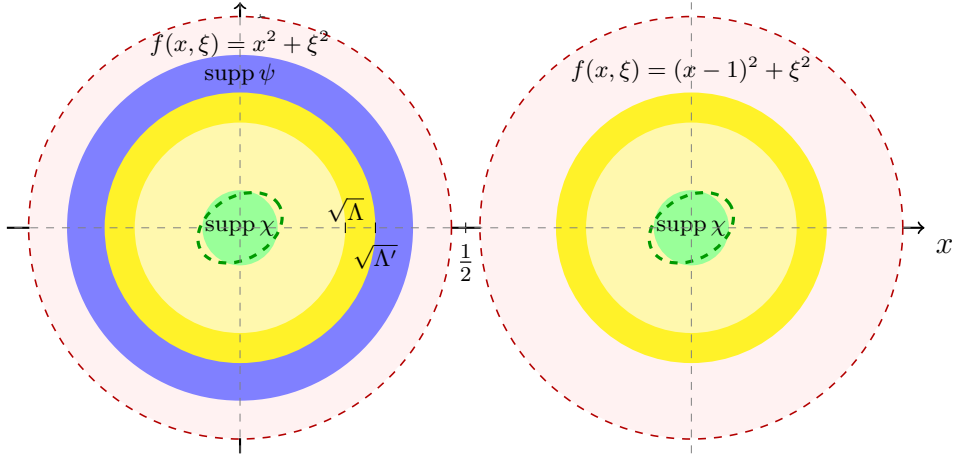
\begin{figure}
\begin{tikzpicture}
\draw[thick,->] (-3.1,0) -- (9.1,0) node[anchor=north west] {$x$};
\draw[thick,->] (0,-3) -- (0,3) node[anchor=north west] {$\xi$};
\draw (3 cm,2pt) -- (3 cm,-2pt) node[anchor=north] {$\frac{1}{2}$};
\draw[red!70!black,very thick,dashed] (0,0) circle (2.8cm);
\fill[red!5!white] (0,0) circle (2.8 cm);
\draw[red!70!black,very thick,dashed] (6,0) circle (2.8cm);
\fill[red!5!white] (6,0) circle (2.8 cm);
\fill[blue!50!white] (0,0) circle (2.3 cm);
\node at (0,2.05) {\scriptsize$\supp\psi$};
\fill[yellow!90!white] (0,0) circle (1.8 cm);
\fill[yellow!90!white] (6,0) circle (1.8 cm);
\fill[yellow!40!white] (0,0) circle (1.4 cm);
\fill[yellow!40!white] (6,0) circle (1.4 cm);
\fill[green!40!white] (0,0) circle (0.5 cm) node[black] {\scriptsize$\supp\chi$};
\fill[green!40!white] (6,0) circle (0.5 cm)node[black] {\scriptsize$\supp\chi$};
\draw[green!60!black,very thick,dashed, rotate=30] (0,0) ellipse (0.6cm and 0.416cm);
\draw[green!60!black,very thick,dashed, rotate around={30:(6,0)}] (6,0) ellipse (0.6cm and 0.416cm);
\draw[step=6cm,gray,very thin, dashed] (-3,-3) grid (9,3);
\draw (1.4 cm,2pt) -- (1.4 cm,-2pt) node[anchor=south] {\scriptsize$\sqrt{\Lambda}$};
\draw (1.8 cm,2pt) -- (1.8 cm,-2pt) node[anchor=north] {\scriptsize$\sqrt{\Lambda'}$};
\node at (0, 2.45) {\scriptsize $f(x,\xi)=x^2+\xi^2$};
\node at (6, 2.1) {\scriptsize $f(x,\xi)=(x-1)^2+\xi^2$};
\end{tikzpicture}
\caption{Nested cutoff functions used in the construction of $\tilde{\mathcal{E}}_N$. The functions $\chi$ and $f$ are defined on the torus, so two fundamental domains are shown. The inner dotted ellipse is the support of $\chi_Q$, which must be contained in the disc $\left\{x^2+\xi^2\le \frac{\Lambda}{2\lambda^2}\right\}$ (see (\ref{PQ}) and Proposition \ref{MainEstimate}).}
\label{f:nestedcutoffs}
\end{figure}
Define
\begin{equation}\label{approxInverseT}
    \tilde{\mathcal{E}}_N(z):=\begin{pmatrix}T^*{E}_QT&0&T^*{E}_{Q,+}\\0&I&0\\{E}_{Q,-}T&0&{E}_{Q,-+} \end{pmatrix}
\end{equation}
and where $$\mathcal{Q}^{-1}\mathcal{E}(z)\mathcal{Q}:=\mathcal{E}_Q(z)=\begin{pmatrix}E_Q&E_{Q,+}\\E_{Q,-}&E_{Q,-+}\end{pmatrix}$$ is the inverse from  Corollary \ref{grushininverse} conjugated by $\mathcal{Q}$ acting on $\One_{P_h\le\Lambda}L^2(\R)\oplus \C^{K+1}$ where we have by Lemma \ref{Grushin Parametrix}, Lemma \ref{NormEquivalence}, and Lemma \ref{twospeccutoffs} that all terms acting on 
$\One_{P_h>\Lambda}L^2(\R)$ are of norm $\Ohi$.
Then we multiply (\ref{grushinT}) and (\ref{approxInverseT}) to get:
\begin{align}
\begin{split}\label{AppInvT}
    &\mathcal{P}_N(z)\tilde{\mathcal{E}}_N(z)\\&=\begin{pmatrix}(I-z\Op_N\chi_Q\wh{M}_N)T^*E_QT+T^*\tilde R_-E_{Q,-}T &S_1&(I-z\Op_N\chi_Q\wh{M}_N)T^*E_{Q,+}+T^*\tilde R_-E_{Q,-+}\\ S_2T^*E_QT&I+S_3&S_2T^*E_{Q,+}\\\tilde R_+TT^*E_QT&0&R_+TT^*E_{Q,-+}\end{pmatrix}\\
    &=\begin{pmatrix}T^*\left((I-z\Op_h\chi_Q\wh{M})E_Q+\tilde R_-E_{Q,-}\right)T &0&T^*\left((I-z\Op_h\chi_Q\wh{M})E_{Q,+}+\tilde R_-E_{Q,-+}\right)\\ 0&I&0\\\tilde R_+E_QT&0&R_+E_{Q,-+}\end{pmatrix}+\Ohi\\
    &:=I+\mathcal{S}(z)
    \end{split}
\end{align}
where $\|\mathcal{S}(z)\|_{H_N\oplus \C^K\to H_N\oplus \C^K}=\Ohi.$

We now resolve the Grushin problem in a similar manner to Corollary \ref{grushininverse}.

\begin{cor}\label{grushininversetorus}
For all $K\in\N$, there is an $r_K$  with $r_K\to\infty$, such that for $h$ small enough there is a $\mathcal{E}_N(z):H_N\oplus \C^{K+1}\to H_N\oplus \C^{K+1}$ defined on $\{|z|\le r_K\}$ such that
\begin{equation*}\mathcal{P}_N(z)\mathcal{E}_N(z)=\mathcal{E}_N(z)\mathcal{P}_N(z)=I.
\end{equation*}
\end{cor}
\begin{proof}
The proof is analogous to that of Corollary \ref{grushininverse}, but simpler as everything is finite-dimensional. We write:
\begin{equation}
\label{InverseT}
    \mathcal{E}_N(z):=\tilde{\mathcal{E}}_N(z)\left(\sum_{j=0}^{\infty}(-1)^j\mathcal{S}(z)^j\right)
\end{equation}
which converges for $h$ small enough by (\ref{AppInvT}) and gives that $\mathcal{E}_N(z)$ is a right inverse. As $H_N$ is a finite-dimensional space, it must also be a left inverse.
\end{proof}

As in Theorem \ref{mainL2}, the Schur complement formula gives:

\begin{proof}[Proof of Theorem \ref{mainHN}]
We have:
\begin{equation*}
    \left(\mathcal{S}^j\right)_{+-}=\Ohi.
\end{equation*} This gives that for $p=\frac{m-K-1}{4}-\epsilon$,
\begin{equation*}
    E_{Q,-+}=-I+zL_K+O(h^p)
    =B_K(-I+zD_K+O(h^p))B_K^{-1},
\end{equation*}
from which the result follows via the same argument as Theorem \ref{mainL2}.
\end{proof}

Finally, we discuss the constraints on the support of $\chi$ for Theorem \ref{mainHN}. We require $\supp\chi,\supp\nu\subseteq\{|x|^2+|\xi|^2\le \Lambda\}$ to use Proposition \ref{MainEstimate}, and in (\ref{PQ}) we replaced $\chi$ with the pulled back cutoff $\chi_Q$. For the theorems on the quantized real line, $\Lambda$ may be fixed arbitrarily large, but on $H_N$ there are conditions imposed by periodicity. Specifically, Lemma \ref{twospeccutoffs} requires $\Lambda'>\|Q\|^2\Lambda$, the construction of the spectral cutoffs $T_0$ require that $\psi=1$ on $\{|x|^2+|\xi|^2\le \Lambda'\}$ and $f(x,\xi)=x^2+\xi^2$ on $\supp\psi$ for Lemmas \ref{AlmostUnitarity} and \ref{AlmostEig} to apply, and $f$ must in particular be a smooth function, so may not equal $x^2+\xi^2$ on a ball about the origin of radius larger than $\frac{1}{2}$. These requirements are satisfied for certain $\psi$, $f$, $\Lambda$, $\Lambda'$ provided $\supp\chi,\supp\nu$ is contained in a ball of radius $\frac{c}{\lambda\|Q\|^2}$ about the origin, as claimed in the discussion in Section \ref{ss:thms}.

\section{Rates of Decay for Nontrapping Eigenvalues}
\label{s:rates}

We now move on to the proof of Theorem \ref{NTDecayRate}. We use the technique of a microlocal weight function, similar to (but simpler than) that used by Nonnenmacher, Sj\"ostrand, and Zworski in \cite{NSZ} and also to Section 6.4 of \cite{DZ}.

\begin{proof}[Proof of Theorem \ref{NTDecayRate}]

Let $g\in\cinf_0(\T^2;\R)$ have $g(x,\xi)=x^2-\xi^2$ on a neighborhood of $\supp\nu$. Let \begin{equation*}M_{tg}:=e^{-t\Op_Ng}\Op_N\nu\wh{M}_Ne^{t\Op_Ng}=e^{-t\Op_Ng}\wh{M}_N\Op_N(M^*\nu)e^{t\Op_Ng}\end{equation*} for some $t$ that may depend on $N$. We shall bound the norm of $M_{tg}$ and optimize by taking $t=K'\log\left(\frac{1}{h}\right)$ at the end. Recall that by results in Section \ref{ss:microlocal} adapted to the calculus on $H_N$, $e^{\pm t\Op_Ng}$ are semiclassical pseudodifferential operators with norm at most $O\left(e^{C|t|}\right)$.
We compute:
\begin{align*}
    &M_{tg}=e^{-t\Op_Ng}\wh{M}_Ne^{t\Op_Ng}e^{-t\Op_Ng}\Op_N(M^*\nu) e^{t\Op_Ng}\\&=e^{-t\Op_Ng}e^{t\Op_N\left(\left(M^{-1}\right)^*g\right)}\wh{M}_N\Op_N\nu_{tg}\\&=e^{-t\Op_Ng}e^{t\Op_N\left(\left(M^{-1}\right)^*g\right)}e^{t\left(\Op_Ng-\Op_N\left(\left(M^{-1}\right)^*g\right)\right)}e^{-t\left(\Op_Ng-\Op_N\left(\left(M^{-1}\right)^*g\right)\right)}\wh{M}_N\Op_N\nu_{tg}.
\end{align*}
where $\Op_N\nu_{tg}=e^{-t\Op_Ng}\Op_N(M^*\nu) e^{t\Op_Ng}$ is a pseudodifferential operator and the second equality comes from the exact Egorov's theorem applied to the functional calculus expression $e^{t\Op_Ng}$. Note also that if $t=O(\log\frac{1}{h})$, then $\|e^{\pm t\Op_Ng}\|_{H_N\to H_N}\le O(h^M)$ for some $M$, so the microlocal support of $\Op_N\nu$ is preserved.

We claim that $e^{-t\Op_Ng}e^{t\Op_N\left(\left(M^{-1}\right)^*g\right)}e^{t\left(\Op_Ng-\Op_N\left(\left(M^{-1}\right)^*g\right)\right)}$ is bounded independent of $t$ and $N$. For convenience let $G_1=\Op_Ng$, $G_2=\Op_N\left(\left(M^{-1}\right)^*g\right)$, so we must bound $W(t):=e^{-tG_1}e^{tG_2}e^{t(G_1-G_2)}$. We have that
\begin{align}\label{WDiffeq}
\begin{split}
    \dd_tW(t)&=-G_1e^{-tG_1}e^{tG_2}e^{t(G_1-G_2)}+e^{-tG_1}e^{tG_2}G_1e^{t(G_1-G_2)}\\&=\left(-G_1+e^{-tG_1}e^{tG_2}G_1e^{-tG_2}e^{tG_1}\right)W(t)\\
    &:=Z(t)W(t)
    \end{split}
\end{align}
where $Z$ is a pseudodifferential operator. To establish bounds on $Z(t)$, we argue as in Theorem 6.21 in \cite{DZ}. We see that
\begin{align}
    \nonumber Z(t)&=-G_1+e^{tG_2}G_1e^{-tG_2}-e^{tG_2}G_1e^{-tG_2}+e^{-tG_1}e^{tG_2}G_1e^{-tG_2}e^{tG_1}\\
    \label{Twoads}&=(-I+e^{\ad_{tG_2}})G_1+(-I+e^{\ad_{-tG_1}})\left(e^{tG_2}G_1e^{-tG_2}\right).
\end{align}
We bound the first term of (\ref{Twoads}), and the second will follow similarly. Fix $K$ large, and expand using Taylor's formula:
\begin{equation}\label{AdExpansion}
    (-I+e^{\ad_{tG_2}})G_1=\sum_{k=1}^K\frac{t^k}{k!}\ad_{G_2}^kG_1+\frac{t^{K+1}}{K!}\int_0^1(1-s)^Ke^{-tsG_2}\ad_{G_2}^{K+1}G_1e^{tsG_2}\,ds.
\end{equation}
Note that by pseudodifferential calculus $\|\ad_{G_2}^kG_1\|_{H_N\to H_N}=O(h^k)$. Plugging into (\ref{AdExpansion}) gives for some $C>0$
\begin{equation*}
   \|(-I+e^{\ad_{tG_2}})G_1\|_{H_N\to H_N}=O(th+(th)^Ke^{Ct})
\end{equation*}
and an analogous argument for $(-I+e^{\ad_{-tG_1}})\left(e^{tG_2}G_1e^{-tG_2}\right)$ gives that
\begin{equation*}
    \|Z(t)\|_{H_N\to H_N}=O(th+(th)^Ke^{Ct}).
\end{equation*}
Now choose $t=K'\log\frac{1}{h}$ for fixed $K'$. Making $K$ large enough, we have $\|Z(s)\|_{H_N\to H_N}=O(h\log\frac{1}{h})$ uniformly in $s\le t$. By (\ref{WDiffeq}) and Gronwall's inequality, we see that \begin{equation*}\|W(t)\|_{H_N\to H_N}\le e^{tCh\log\frac{1}{h}}=e^{C_0h\left(\log\frac{1}{h}\right)^2}\le C_1.\end{equation*}

We now show the norm of $e^{-t\left(\Op_Ng-\Op_N\left(\left(M^{-1}\right)^*g\right)\right)}\wh{M}_N\Op_N\nu_{tg}$ to be small as $t$ grows large. Note that $g-\left(\left(M^{-1}\right)^*g\right)\ge c_0$ on a neigborhood away from the origin, which we may take as $\supp\nu_{tg}$ up to $\Ohi$ errors, so by the same method as Lemma 3.6 in \cite{NSZ} we may replace it with $\left(g-\left(\left(M^{-1}\right)^*g\right)\right)\nu_1+(1-\nu_1)$ with $\nu_1=1$ on $\supp\nu$ and $g=x^2-\xi^2$ on $\supp\nu_1$. Then by the G\aa rding inequality, $\Op_N(g-(M^{-1})^*g)\nu_1+(1-\nu_1)$ has spectrum bounded below by $\frac{c_0}{2}$ for $h$ small enough. This gives \begin{equation*}\left\|e^{-t\left(\Op_Ng-\Op_N\left(\left(M^{-1}\right)^*g\right)\right)}\right\|
\le e^{-\frac{t}{2}c_0}=h^{K'\frac{c_0}{2}}
\end{equation*}
where $K'$ can be chosen arbitrarily large, which completes the proof.
\end{proof}

\section{Acknowledgements}
We would like to thank Maciej Zworski for suggesting the problem and giving guidance along the way. We also acknowledge the partial support under the NSF grant DMS-1952939.

\bibliographystyle{acm}

\bibliography{sources}

\end{document}